\documentclass[12pt]{article}
\RequirePackage{etex}
\usepackage{amsmath}
\usepackage{graphicx}
\usepackage{enumerate}
\usepackage{natbib}
\usepackage{url}
\usepackage{bm}
\usepackage{bbm}
\usepackage{amsthm,amssymb,amsfonts}
\usepackage{subfigure}
\usepackage{float}
\usepackage{xcolor}
\usepackage{multirow}
\usepackage{ulem}
\usepackage{multirow}
\usepackage{comment}
\RequirePackage[colorlinks,citecolor=blue,urlcolor=blue,breaklinks=true]{hyperref}
\usepackage{mathtools}
\usepackage{booktabs}
\usepackage{cleveref, autonum}
\usepackage{rotating}
\usepackage[margin = 1in]{geometry}
\usepackage[title]{appendix}

\newcommand{\Var}{\mathrm{Var}}

\newcommand{\tr}{\mathrm{tr}}
\newcommand\data{\mathbb{D}_n}
\newcommand{\EE}{\mathbb{E}}
\newcommand{\PP}{\mathbb{P}}
\newcommand{\bbR}{\mathbb{R}}
\newcommand{\bbH}{\mathbb{H}}
\newcommand{\mX}{\mathcal{X}}
\newcommand{\s}{\sum_{i=1}^{\infty}}
\newcommand{\Ltwo}{L^2_{p_X}(\mX)}
\newcommand{\GP}{\mathrm{GP}}

\theoremstyle{plain}
\newtheorem{theorem}{Theorem}
\newtheorem{lemma}[theorem]{Lemma}

\newtheorem{corollary}[theorem]{Corollary}

\newtheorem{assumption}{Assumption}
\renewcommand{\theassumption}{A}

\newtheoremstyle{nonitalicremark}
  {} 
  {} 
  {} 
  {} 
  {\bfseries} 
  {.} 
  {.5em} 
  {} 

\theoremstyle{nonitalicremark}
\newtheorem*{remark}{Remark}

\graphicspath{{figures/}}

\begin{document}

\def\spacingset#1{\renewcommand{\baselinestretch}
{#1}\small\normalsize} \spacingset{1}

\title{\bf Optimal plug-in Gaussian processes for modeling derivatives}
\author{Zejian Liu and Meng Li\\
Department of Statistics, Rice University}
\date{}
\maketitle

\bigskip

\begin{abstract}
Derivatives are a key nonparametric functional in wide-ranging applications where the rate of change of an unknown function is of interest. In the Bayesian paradigm, Gaussian processes (GPs) are routinely used as a flexible prior for unknown functions, and are arguably one of the most popular tools in many areas. However, little is known about the optimal modeling strategy and theoretical properties when using GPs for derivatives. In this article, we study a plug-in strategy by differentiating the posterior distribution with GP priors for derivatives of any order. This practically appealing plug-in GP method has been previously perceived as suboptimal and degraded, but this is not necessarily the case. We provide posterior contraction rates for plug-in GPs and establish that they achieve optimal rates simultaneously for all derivative orders. We show that the posterior measure of the regression function and its derivatives, with the same choice of hyperparameter that does not depend on the order of derivatives, converges at the minimax optimal rate up to a logarithmic factor for functions in certain classes. We analyze a data-driven hyperparameter tuning method based on empirical Bayes, and show that it satisfies the optimal rate condition while maintaining computational efficiency. This article to the best of our knowledge provides the first positive result for plug-in GPs in the context of inferring derivative functionals, and leads to a practically simple nonparametric Bayesian method with optimal and adaptive hyperparameter tuning for simultaneously estimating the regression function and its derivatives. Simulations show competitive finite sample performance of the plug-in GP method. A climate change application for analyzing the global sea-level rise is discussed. 
\end{abstract}

\vfill

\newpage

\section{Introduction}

\sloppy

Consider the nonparametric regression model 
\begin{equation}\label{eq:model}
Y_i=f(X_i)+\varepsilon_i,\quad \varepsilon_i\sim N(0,\sigma^2),
\end{equation}
where the data $\data=\{X_i,Y_i\}_{i=1}^{n}$ are independent and identically distributed samples from a distribution $\PP_0$ on $\mX\times \bbR$ that is determined by $\PP_X$, $f_0$, and $\sigma^2$, which are respectively the marginal distribution of $X_i$, the true regression function, and the noise variance that is possibly unknown. Let $p_X$ denote the density of $\PP_X$ with respect to the Lebesgue measure $\mu$. Here $\mX\subset \bbR^p$ is a compact metric space for $p\geq 1$.

We are interested in the inference on the derivative functions of $f$. Derivatives emerge as a key nonparametric quantity when the rate of change of an unknown surface is of interest. Examples include surface roughness for digital terrain models, temperature or rainfall slope in meteorology, and pollution curvature for environmental data. The importance of derivatives, either as a nonparametric functional or localized characteristic of $f$, can also be found in efficient modeling of functional data \citep{dai2018derivative}, shape constrained function estimation~\citep{riihimaki2010gaussian, wang2016estimating}, and the detection of stationary points~\citep{yu2020bayesian}. Optimal convergence rates for derivatives under Hilbert and Sobolev norms for kernel-based regularized least-squares estimators are provided in \citep{blanchard2018optimal,fischer2020sobolev}.

Gaussian processes (GPs) are a popular nonparametric Bayesian method in many areas such as spatially correlated data analysis \citep{stein2012interpolation,gelfand2003spatial,banerjee2003directional}, functional data analysis \citep{shi2011gaussian}, and machine learning \citep{rasmussen2006}; see also the excellent review article by \cite{gelfand2016spatial} which elaborates on the instrumental role GPs have played as a key ingredient in an extensive list of twenty years of modeling work. GPs not only provide a flexible process for unknown functions but also serve as a building block in hierarchical models for broader applications. 

For function derivatives, the so-called plug-in strategy that directly differentiates the posterior distribution of GP priors is practically appealing, as it would allow users to employ the same prior no matter whether the inference goal is on the regression function or its derivatives. However, this plug-in estimator has been perceived as suboptimal and degraded for a decade \citep{stein2012interpolation,holsclaw2013gaussian} based on heuristics, while a theoretical understanding is lacking, particularly in light of the nonparametric nature of derivative functionals. As a result, there has been limited study of plug-in GPs ever since, and substantially more complicated methods that hamper easy implementation and often restrict to one particular derivative order are pursued. 

In this article, we study the plug-in strategy with GPs for derivative functionals by characterizing large sample properties of the plug-in posterior measure with GP priors, and obtain the first positive result. We show that the plug-in posterior distribution, with the same choice of hyperparameter in the GP prior, concentrates at the derivative functionals of any order at a nearly minimax rate in specific examples, thus achieving a remarkable plug-in property for nonparametric functionals that gains increasing attention recently \citep{yoo2016supremum,liu2020non}. It is known that many commonly used nonparametric methods such as smoothing splines and local polynomials do not enjoy this property when estimating derivatives \citep{wahba1990optimal,charnigo2011generalized}, and the only nonparametric Bayesian method with established plug-in property, to the best of our knowledge, is random series priors with B-splines \citep{yoo2016supremum}. 

In recent years, the nonparametric Bayesian literature has seen remarkable adaptability of GP priors in various regression settings~\citep{van2009adaptive, bhattacharya2014anisotropic, jiang2021variable}. Our findings contribute to this growing literature and indicate that the widely used GP priors offer more than inferring regression functions. In particular, the established theory reassures the use of plug-in GPs for optimal modeling of derivatives, and further sheds light on hyperparameter tuning in the presence of varying derivative orders, for which we propose to use an empirical Bayes approach. Our analysis indicates that this data-driven hyperparameter tuning strategy attains theoretical optimality for all derivative orders and adapts to the true function's smoothness level with an oversmooth kernel, while maintaining computational efficiency. Therefore, this article shows that the Bayes procedure using GP priors automatically achieves optimal contraction rates for derivatives of all orders, leading to a practically simple nonparametric Bayesian method with guided hyperparameter tuning for simultaneously estimating the regression function and its derivatives. These theoretical guarantees are complemented by competitive finite sample performance using simulations, as well as a climate change application to analyzing the global sea-level rise. 

The following notation is used throughout this paper. We write $X=(X_1^T,\ldots,X_n^T)^T\in\bbR^{n\times p}$ and $Y=(Y_1,\ldots,Y_n)^T\in \bbR^n$. Let $\|\cdot\|$ be the Euclidean norm; for $f, g: \mX \rightarrow \bbR$, let $\|f\|_\infty$ be the $L_\infty$ (supremum) norm, $\|f\|_2=(\int_\mX f^2d\PP_X)^{1/2}$ the $L_2$ norm with respect to the covariate distribution $\PP_X$, and $\left<f,g\right>_2=(\int_\mX fgd\PP_X)^{1/2}$ the inner product. The corresponding $L_2$ space relative to $\PP_X$ is denoted by $\Ltwo$; we write $L^2(\mX)$ for the $L_2$ space with respect to the Lebesgue measure $\mu$. Let  $\mathbb{N}$ be the set of positive integers and write $\mathbb{N}_0=\mathbb{N}\cup\{0\}$. We let $C(\mX)$ and $C(\mX,\mX)$ denote the space of continuous functions and continuous bivariate functions. In one-dimensional case, for $\Omega\subset\bbR$, a function $f:\Omega\rightarrow\bbR$ and $k\in\mathbb{N}$, we use $f^{(k)}$ to denote its $k$-th derivative when it exists, with $f^{(0)}=f$. Let $C^m(\Omega)=\{f:\Omega\rightarrow\bbR \mid f^{(k)}\in C(\Omega) \ {\rm for\ all\ } 1\leq k\leq m\}$ denote the space of $m$-times continuously differentiable functions, and $C^{2m}(\Omega,\Omega)=\{K:\Omega\times \Omega\rightarrow\bbR \mid \partial^k_x\partial^k_{x'}K(x,x')\in C(\Omega,\Omega) \ {\rm for\ all\ } 1\leq k\leq m\}$ the space of $m$-times continuously differentiable bivariate functions, where $\partial^k_x=\partial^k/\partial x^k$. For two sequences $a_n$ and $b_n$, we write $a_n \lesssim b_n$ if $a_n \leq C b_n$ for a universal constant $C>0$, and $a_n \asymp b_n$ if $a_n \lesssim b_n$ and $b_n \lesssim a_n$.

\section{Main results}\label{sec:main.results}

\subsection{Plug-in Gaussian process for derivative functionals}

We assign a Gaussian process prior $\Pi$ on the regression function such that $f \sim \GP(0, \sigma^2 (n \lambda)^{-1} K)$. Here $K(\cdot,\cdot):\mX\times\mX\rightarrow \bbR$ is a continuous, symmetric and positive definite kernel function, and $\lambda>0$ is a regularization parameter that possibly depends on the sample size $n$. The rescaling factor $\sigma^2(n\lambda)^{-1}$ in the covariance kernel connects the posterior mean Bayes estimator with kernel ridge regression \citep{wahba1990spline,cucker2007learning}; see also Theorem 11.61 in \cite{ghosal2017fundamentals}. This rescaling strategy has been employed in Bayesian inverse problems~\citep{nickl2023bayesian}, while alternative approaches fix $\lambda$~\citep{van2008rates} and instead rescale sample paths via a bandwidth parameter~\citep{van2009adaptive} to achieve adaptivity.

It is not difficult to derive that the posterior distribution $\Pi_n(\cdot\mid\data)$ is also a GP: $f \mid \data \sim \GP(\hat{f}_n, \hat{V}_n)$ \cite[Chapter~2.2]{rasmussen2006}, where the posterior mean $\hat{f}_n$ and posterior covariance $\hat{V}_n$ are given by
\begin{align}
\hat{f}_n(x) &= K(x, X) [K(X, X) + n \lambda I_n]^{-1} Y,\\
\label{eq:variance}\hat{V}_n(x, x') &=  \sigma^2 (n \lambda)^{-1} \{K(x, x') - K(x, X)[K(X, X) + n \lambda I_n]^{-1} K(X, x')\},
\end{align}
for any $x, x' \in \mX$. Here $K(X, X)$ is the $n$ by $n$ matrix $(K(X_i, X_j))_{i, j = 1}^n$, $K(x, X)$ is the 1 by $n$ vector $(K(x, X_i))_{i = 1}^n$, and $I_n$ is the $n$  by $n$ identity matrix. 

We now define the plug-in Gaussian process for differential operators. For simplicity we focus on the one-dimensional case where $\mX=[0,1]$ throughout the paper, but remark that the studied plug-in strategy can be extended to multivariate cases straightforwardly despite more complicated notation for high-order derivatives. 

For any $k\in\mathbb{N}$, define the $k$-th differential operator $D^k:C^k[0,1]\rightarrow C[0,1]$ by $D^k(f)=f^{(k)}$. If $K\in C^{2k}([0,1],[0,1])$, then the posterior distribution of the derivative $f^{(k)}\mid \data$, denoted by $\Pi_{n,k}(\cdot\mid\data)$, is also a GP since differentiation is a linear operator \cite[Chapters~4.1.1 and 9.4]{rasmussen2006}. In particular, $f^{(k)} \mid \data \sim {\rm GP}(\hat{f}^{(k)}_n, \tilde{V}^k_n)$, where
\begin{align}
\label{eq:f.hat.deriv}\hat{f}^{(k)}_n(x) &= K_{k0}(x, X) [K(X, X) + n \lambda I_n]^{-1} Y,\\
\label{eq:deriv.variance}\tilde{V}^k_n(x, x')&= \sigma^2 (n \lambda)^{-1} \left\{K_{kk}(x, x') -  K_{k0}(x, X)[K(X, X) + n \lambda I_n]^{-1} K_{0k}(X, x')\right\},
\end{align}
with $K_{k0}(x, X) = (\partial^k_x K(x, X_i))_{i = 1}^n$ and $K_{kk}(x, x')=\partial^{k}_x\partial^k_{x'}K(x,x')$. Then the nonparametric plug-in procedure for $D^k$ refers to using the plug-in posterior measure $\Pi_{n,k}(\cdot\mid\data)$ for inference on $D^k(f)$. 

The plug-in posterior measure $\Pi_{n,k}(\cdot\mid\data)$ has a closed-form expression with given $\lambda$ and $\sigma^2$, substantially facilitating its implementation in practice. The plug-in strategy is practically appealing but has been perceived as suboptimal for a decade \citep{stein2012interpolation,holsclaw2013gaussian} based on heuristics. To the contrary, we will establish the optimality of plug-in GPs and uncover their adaptivity to derivative orders. In particular, our large sample analysis of $\Pi_{n,k}(\cdot\mid\data)$ shows that GPs exhibit a ``nonparametric plug-in property'' for derivative functionals, a concept we introduce and contextualize in the next section.

\subsection{Nonparametric plug-in property for derivative functionals} 
\label{sec:plug.in.challenge}

The ``plug-in property'' in the literature~\citep{bickel2003nonparametric} refers to the phenomenon that a rate-optimal nonparametric estimator also efficiently estimates some bounded linear functionals. A parallel concept has been studied in the Bayesian paradigm relying on posterior distributions and posterior contraction rates~\citep{rivoirard2012bernstein, castillo2013nonparametric, castillo2015bernstein}. 

The differential operator $D^k$ studied in this article, however, points to function-valued functionals, or \textit{nonparametric functionals}, as opposed to real-valued functionals studied in the classical plug-in property literature. Hence, one needs to analyze the function-valued functionals uniformly for all points in the support. To distinguish the plug-in property for nonparametric functionals from its traditional counterpart for real-valued functionals, we term this property as \textit{nonparametric plug-in property}. 

The nonparametric plug-in property for derivative estimation of a regression function is known to fail for many existing estimators. For instance, as highlighted by \cite{wahba1990optimal} and \cite{charnigo2011generalized}, the optimal selection of the smoothing parameter in widely used methods such as local polynomial regression and smoothing splines is contingent on the order of the derivative. Specifically, when the estimator attains the optimal rate of convergence for the regression function, derivative estimators with the same value of tuning parameter are sub-optimal, and vice versa. This lack of adaptability to varying derivative orders complicates the process of parameter tuning in these methods.

We will establish that GPs enjoy the nonparametric plug-in property for differential operators. The theoretical development relies on the explicit Gaussian characterization of the posterior distribution over derivatives. Posterior contraction rates are derived through a detailed analysis of the posterior variance and the convergence properties of the posterior mean estimator. This general strategy bears resemblance to prior work in Bayesian nonparametrics, such as \cite{knapik2011bayesian}, although the focus and inferential objectives differ from those considered here. In the case of derivative functionals, our analysis draws on an operator-theoretic framework \citep{smale2005shannon,smale2007learning}, the equivalent kernel technique, and recent advances in non-asymptotic analysis of nonparametric estimators \citep{liu2020non}. Our theory is developed under a random design framework, which includes quasi-uniform (or quasi-fixed) designs where inputs are randomly sampled from a finite set, such as equally spaced grid points. A full extension to other designs, such as fixed design, is an interesting direction for future work.

\subsection{Posterior contraction for function derivatives}

Throughout this article, we assume the true regression function $f_0 \in C^k[0,1]$ and the covariance kernel $K\in C^{2k}([0,1],[0,1])$. Let $\{\mu_i\}_{i=1}^\infty$ and $\{\phi_i\}_{i=1}^\infty$ be the eigenvalues and eigenfunctions of the kernel $K$ such that $K(x,x')=\sum_{i=1}^{\infty}\mu_i\phi_i(x)\phi_i(x')$ for any $x,x'\in [0,1]$, where the eigenvalues satisfy $\mu_1\geq\mu_2\geq \cdots> 0$ and $\mu_i \downarrow 0$, and eigenfunctions form an orthonormal basis  of $L_{p_X}^2[0,1]$. The existence of such eigendecomposition is ensured by Mercer's theorem \cite[Theorem~4.2]{rasmussen2006}. It can also be seen that $\phi_i\in C^{k}[0,1]$ for all $i\in\mathbb{N}$ as $K\in C^{2k}([0,1],[0,1])$.

We make the following assumptions on the eigenfunctions of the covariance kernel.

\renewcommand{\theassumption}{(A)}

\begin{assumption}
There exists $C_{k,\phi}>0$ such that $\|\phi^{(k)}_i\|_\infty \leq C_{k,\phi} i^k$ for any $i\in\mathbb{N}$.
\end{assumption}

\renewcommand{\theassumption}{(B)}

\begin{assumption}
There exists $L_{k,\phi}>0$ such that $|\phi^{(k)}_i(x)-\phi^{(k)}_i(x')| \leq L_{k,\phi} i^{k+1} |x-x'|$ for all $i\in\mathbb{N}$ and any $x,x'\in[0,1]$.
\end{assumption}

We will make extensive use of the so-called equivalent kernel $\tilde{K}$ \cite[Chapter~7]{rasmussen2006}, which shares the same eigenfunctions with $K$ with altered eigenvalues $\nu_i={\mu_i}/{(\lambda+\mu_i)}$ for $i\in\mathbb{N}$, i.e., $\tilde{K}(x,x')=\s \nu_i\phi_i(x)\phi_i(x')$. Note that $\tilde{K}$ is also a continuous, symmetric, and positive definite kernel.

Under Assumption (A), we define an $m$-th order analog of \textit{effective dimension} of the kernel $K$ with respect to $p_X$ for $0\leq m\leq k$:
\begin{equation}\label{eq:high.order.kappa}
\tilde{\kappa}_{m}^2=\sup_{x,x'\in[0,1]}\partial_{x}^m\partial_{x'}^m\tilde{K}(x,x')=\sup_{x\in[0,1]}\s\frac{\mu_i}{\lambda+\mu_i}\phi^{(m)}_i(x)^2.
\end{equation}
When $m = 0$, the definition above recovers the effective dimension of $K$ in the literature~\citep{zhang2005learning}, namely,  $\tilde{\kappa}^2=\tilde{\kappa}_{0}^2 =\sup_{x\in [0,1]}\tilde{K}(x,x) \lesssim\sum_{i=1}^{\infty}\frac{\mu_i}{\lambda+\mu_i}$.
We also define $\hat{\kappa}_{m}^2$ as a useful upper bound of $\tilde{\kappa}_{m}^2$ up to a constant multiplier under Assumption (A) for $0 \leq m\leq k$:
\begin{equation}\label{eq:kappa.k+1}
\hat{\kappa}_{m}^2=\s\frac{i^{2m}\mu_i}{\lambda+\mu_i}\gtrsim \tilde\kappa_{m}^2.
\end{equation}
Note that $\tilde\kappa_{m}^2$ and $\hat\kappa_{m}^2$ depend on the regularization parameter $\lambda$.

Let $\epsilon_n$ be a sequence such that $\epsilon_n\rightarrow0$, and $n\epsilon_n^2\rightarrow\infty$. For $\epsilon_n$ to become a posterior contraction rate for $\Pi_{n,k}(\cdot\mid\data)$ at $f_0^{(k)}$, we additionally require the following assumptions.

\renewcommand{\theassumption}{(C)}

\begin{assumption}
The regularization parameter $\lambda$ is chosen such that $\tilde{\kappa}^2=o(\sqrt{n/\log n})$, $\tilde{\kappa}_{k}^2=O(n\epsilon_n^2/\log n)$, and $\hat{\kappa}_{k+1}^2=O(n)$.
\end{assumption}

\renewcommand{\theassumption}{(D)}

\begin{assumption}
$\epsilon_n$ is a non-asymptotic convergence rate of $\hat{f}_n^{(k)}$ under the $L_2$ or $L_\infty$ norm, i.e., $\|\hat f_n^{(k)}-f_0^{(k)}\|_p\lesssim \epsilon_n$ for $p = 2$ or $\infty$ with $\PP_0^{(n)}$-probability tending to 1.
\end{assumption}

Assumptions (A), (B), and (C) are related to the eigendecomposition of the covariance kernel, while Assumption (D) is on the convergence rate of the derivatives of the posterior mean, which are well studied in nonparametric statistics. Assumption (C) can be verified through direct calculations based on the decay rate of eigenvalues; we provide the rates of $\tilde{\kappa}^2$, $\tilde{\kappa}_{k}^2$ and $\hat{\kappa}_{k+1}^2$ for specific kernels later in the paper in Lemma~\ref{lem:differentiability.exp} and Lemma~\ref{lem:differentiability.matern}. In view of the connection between the posterior mean and kernel ridge regression (KRR) estimator, we can take advantage of the rich literature on KRR and kernel learning theory to verify Assumption (D). Along this line, Section~\ref{sec:Matern} will give the convergence rates for special examples in Lemma~\ref{thm:minimax.diff.exp} and Lemma~\ref{thm:minimax.diff.matern}. 

The following theorem obtains the posterior contraction rates using plug-in GPs for function derivatives.

\begin{theorem}\label{thm:deriv.contraction.l2}
Suppose that $f_0 \in C^k[0,1]$, $K\in C^{2k}([0,1],[0,1])$ and $\epsilon_n$ is a sequence such that $\epsilon_n\rightarrow0$, and $n\epsilon_n^2\rightarrow\infty$. Under Assumptions (A), (B), (C) and (D), $\epsilon_n$ is a posterior contraction rate of $\Pi_{n,k}(\cdot \mid \data)$ at $f_0^{(k)}$, i.e., for any $M_n\rightarrow\infty$, $\Pi_{n,k}\left(\|f^{(k)}-f_0^{(k)}\|_p > M_n\epsilon_n\,\big|\,\data\right)\rightarrow 0$ in $\PP_0^{(n)}$-probability.
\end{theorem}

Theorem~\ref{thm:deriv.contraction.l2} provides an approach for establishing posterior contraction rates. Computing the contraction rate now boils down to analyzing the eigendecomposition of the  covariance kernel (Assumptions (A), (B), and (C)) and the convergence rate of the derivatives of the posterior mean (Assumption (D)). In Section~\ref{sec:Matern}, we will apply Theorem~\ref{thm:deriv.contraction.l2} to derive the minimax optimal contraction rate for estimating the derivatives of analytic-type functions (Theorem~\ref{thm:exp.contraction.deriv}) and H\"older smooth functions (Theorem~\ref{thm:contraction.deriv}). Although our focus is on function derivatives, setting $k = 0$ in Theorem~\ref{thm:deriv.contraction.l2} also provides contraction rates for estimating the regression function.

\subsection{Minimax optimality using special examples}\label{sec:Matern}

In this section, we will provide examples in which optimal contraction rates in the minimax sense can be achieved by using Theorem~\ref{thm:deriv.contraction.l2}. A technical requirement is to verify Assumption D, which pertains to
the convergence rate of derivatives of the posterior mean. To this end, we first introduce a useful result for any kernel with bounded eigenfunctions to compute the convergence rate required by Assumption (D). 

We now briefly review the preliminaries for reproducing kernel Hilbert space (RKHS), commonly used in machine learning and statistical learning theory. The material follows standard references such as \cite{rasmussen2006,berlinet2011reproducing}, and is included here for completeness and to introduce notation.

Consider $\mathbb{H}$ as the RKHS associated with the kernel $K$. For representations $f = \sum f_i\phi_i$ and $g = \sum g_i\phi_i$, the inner product in $\mathbb{H}$ is given by $\langle f,g \rangle_{\mathbb{H}} = \sum f_ig_i/\mu_i$. Additionally, we introduce $\tilde{\mathbb{H}}$ to represent the RKHS associated with the corresponding kernel $\tilde{K}$. Note that while $\tilde{\mathbb{H}}$ and $\mathbb{H}$ comprise the same functions, their inner products differ. In $\tilde{\mathbb{H}}$, the inner product is defined as $\langle f,g \rangle_{\tilde{\mathbb{H}}} = \langle f,g \rangle_2 + \lambda\langle f,g \rangle_\mathbb{H}$. 

Further, we introduce a compact, positive definite, and self-adjoint integral operator $L_K$ from $\Ltwo$ to $\mathbb{H}$ described by:
\begin{equation} \label{eq:int.operator}
L_K(f)(x) = \int_{\mX} K(x, x') f(x') d\PP_X(x'), \quad x \in \mX. 
\end{equation}
Let $I$ be the identity operator. We approximate $f_0$ by a function in $\bbH$:
\begin{equation}\label{eq:flambda}
f_{\lambda} = (L_K + \lambda I)^{-1}L_K f_{0},
\end{equation}
which minimizes $ \|f - f_{0}\|_{2}^2 + \lambda \|f\|^2_{\bbH}$ subject to $f \in \bbH$. The equivalent kernel $\tilde{K}$ provides an alternative interpretation of the proximate function $f_\lambda$. Letting the integral operator $L_{\tilde{K}}$ be the counterpart of $L_K$ induced by $\tilde{K}$, we have $f_\lambda = L_{\tilde{K}}f_0$. 

The following Lemma~\ref{lem:approx.error} facilitates the verification of  Assumption (D) by relating the convergence rate of the posterior mean to the approximate function $f_\lambda$. In particular, the convergence rate can be expressed as the sum of two terms: a deterministic term independent of the data or error, and a term determined by the equivalent RKHS induced by the covariance kernel.

\begin{lemma}\label{lem:approx.error}
Under Assumption (A), suppose that $\|f^{(k)}\|_2\leq D(k,\lambda) \|f\|_{\tilde{\bbH}}$ for any $f\in\bbH$ and some $D(k,\lambda) > 0$. By choosing $\lambda$ such that $\tilde{\kappa}^2=o(\sqrt{n/\log n})$ and $\|f_\lambda - f_0\|_\infty = o(1)$, we have with $\PP_0^{(n)}$-probability at least $1-n^{-10}$ that
\begin{equation}
\|\hat{f}_n^{(k)}-f_0^{(k)}\|_2 \lesssim\|f_\lambda^{(k)}-f_0^{(k)}\|_2 + D(k,\lambda) \tilde\kappa \sqrt{\frac{\log n}{n}}.
\end{equation}
\end{lemma}

Next, we present two concrete examples as a direct application of Theorem~\ref{thm:deriv.contraction.l2}, where we use kernels with exponentially decaying and polynomially decaying eigenvalues, respectively, for estimating analytic-type functions and Sobolev and H\"older functions. We consider a uniform sampling process for $p_X$ on $[0,1]$, and the corresponding probability measure $\PP_X$ becomes the Lebesgue measure. For concreteness, the eigenfunctions are the Fourier basis functions
\begin{equation}\label{eq:fourier}
\psi_1(x)=1,\ \psi_{2i}(x)=\sqrt{2}\cos(2\pi ix),\ \psi_{2i+1}=\sqrt{2}\sin(2\pi ix),\quad i\in\mathbb{N}, 
\end{equation}
where we use $\phi_i$ to denote general eigenfunctions and $\psi_i$ to refer to the Fourier basis functions considered in our special examples. Such eigenfunctions easily satisfy Assumption (A) with $C_{k,\psi}=\sqrt{2}(2\pi)^k$ and Assumption (B) with $L_{k,\psi} = \sqrt{2}(2\pi)^{k+1}$, and enable some explicit calculations when verifying Assumption D by Lemma~\ref{lem:approx.error}. The Fourier basis system is commonly used in the literature (e.g., \cite{van2011information} and \cite{yang2017frequentist}); they are also the eigenfunctions of the squared exponential kernel restricted to the circle~\citep{Li+Ghosal:17}. Note that the general results established in the preceding sections do not impose these constraints on the marginal distribution of $X$ or eigenfunctions.

We first consider the analytic-type function class $A^\gamma[0,1]$ for the true regression function $f_0$:
\begin{equation}
A^\gamma[0,1]=\left\{f: \|f\|^2_{A^\gamma[0,1]}= \s e^{2\gamma i}f_i^2<\infty,f_i=\left<f,\psi_i\right>_2\right\},
\end{equation}
where the eigenfunctions are the Fourier basis functions. It can be seen that $A^\gamma[0,1] \subset C^\infty[0,1]$ and the smoothness level of functions in $A^\gamma[0,1]$ increases as $\gamma$ increases. This function class has also been considered in \cite{van2011information}.

We use kernels with exponentially decaying eigenvalues, i.e., the covariance kernel $K$ has an eigendecomposition relative to $\PP_X$ with eigenvalues $\mu_i\asymp e^{-2\gamma i}$ for $i\in\mathbb{N}$ and some $\gamma>0$. We denote such kernels by $K_\gamma$.  The well-known squared exponential kernel exhibits similarly exponential eigenvalue decay, with eigenfunctions related to Hermite polynomials under a Gaussian design for $p_X$ \cite[Chapter~7]{pati2015adaptive,rasmussen2006}.

By verifying Assumptions (C) and (D) using our non-asymptotic analysis in Section~\ref{sec:non-asymptotic}, Theorem~\ref{thm:deriv.contraction.l2} yields the nonparametric plug-in property for estimating the derivatives of analytic-type functions: the posterior distribution contracts at a nearly parametric rate under the $L_2$ norm.

\begin{theorem}\label{thm:exp.contraction.deriv}
Suppose $f_0\in A^\gamma[0,1]$ for $\gamma>1/2$. If $K_\gamma$ is used in the GP prior with the regularization parameter $\lambda\asymp \log n/n$, then for any $k\in\mathbb{N}_0$, the posterior distribution $\Pi_{n,k}(\cdot\mid\data)$ contracts at $f_0^{(k)}$ at a nearly parametric rate $\epsilon_n = (\log n)^{k+1}/\sqrt{n}$ under the $L_2$ norm.
\end{theorem}

We next consider a Sobolev space 
\begin{equation}
W^\alpha[0,1] = \left\{f:  \|f\|^2_{W^\alpha[0,1]} = \s i^{2\alpha}f_i^2 < \infty,f_i=\left<f,\psi_i\right>_2\right\}, 
\end{equation}
and a H\"older space
\begin{equation}
H^{\alpha}[0,1]=\left\{f: \|f\|^2_{H^{\alpha}[0,1]}=\sum_{i=1}^{\infty}i^{\alpha}|f_i|<\infty,f_i=\left<f,\psi_i\right>_2\right\},
\end{equation}
for $\alpha > 0$; here $\alpha$ specifies the degree of smoothness. Although we use the Fourier basis as $\psi_i$, these two spaces are defined using the sequence $f_i$ and can be generalized to other eigenfunctions $\psi_i$. For any eigenfunctions satisfying Assumptions (A) and (B), functions in $H^{\alpha}[0,1]$ have continuous derivatives up to order $\lfloor\alpha\rfloor$, with the $\lfloor\alpha\rfloor$th derivative being H\"older continuous of order $\alpha-\lfloor\alpha\rfloor$; see~\cite{yang2017frequentist,liu2020non} for a similar discussion. With the Fourier basis as $\psi_i$ and $\alpha \in \mathbb{N}$, $W^\alpha[0,1]$ is the so-called \textit{periodic Sobolev space} (see Theorem 7.11 in \cite{wasserman2006all}) and can be equivalently described as $W^\alpha[0,1] = \left\{f:  f \in C^{\alpha-1}[0,1],  f^{(\alpha)} \in L^2[0,1],
f^{(j)}(0) = f^{(j)}(1), j = 0,\ldots, \alpha-1 \right\}.$ Periodic function spaces are commonly used when solving partial differential equations (PDEs), with periodicity being crucial in specific contexts when dealing with problems exhibiting periodic or oscillatory behavior~\citep{sani1994resume,dong2014robust}. Beyond the wide-ranging applications of PDEs in computational sciences, including fluid mechanics \citep{dong2006combined, dong2008turbulent}, the periodic Sobolev space has also been considered in nonparametric inference for smoothing spline estimation~\citep{nussbaum1985spline, sc13aos}. 

We consider kernels with polynomially decaying eigenvalues $K_\alpha$, where
$
\mu_i\asymp i^{-2\alpha}$ for $i\in\mathbb{N} 
$
and the eigenfunctions are the Fourier basis $\{\psi_i\}_{i=1}^\infty$. Examples of such kernels include the periodic Mat\'ern kernel \citep{giordano2022nonparametric}. It is well known that the eigenvalues of the classical Mat\'ern kernel with parameter satisfy $\mu_i \asymp i^{-2(\nu+1/2)}$ for $i \in \mathbb{N}$ \citep{seeger2008information,santin2016approximation}, while the eigenfunctions do not directly relate to the Fourier basis.

Verifying Assumptions (C) and (D), and invoking Theorem~\ref{thm:deriv.contraction.l2}, we obtain the nonparametric plug-in property for estimating the derivatives of functions in the H\"older class: the posterior distribution contracts at a nearly minimax optimal rate under the $L_2$ norm \citep{stone1982optimal}.

\begin{theorem}\label{thm:contraction.deriv}
Suppose $f_0 \in W^\alpha[0,1]$ or $f_0\in H^\alpha[0,1]$ for $\alpha>3/2$. If $K_\alpha$ is used in the GP prior with the regularization parameter $\lambda\asymp ({\log n}/{n})^{\frac{2\alpha}{2\alpha+1}}$, then for any $k<\alpha-3/2$ and $k\in\mathbb{N}_0$, the posterior distribution $\Pi_{n,k}(\cdot\mid\data)$ contracts at $f_0^{(k)}$ at the nearly minimax optimal rate $\epsilon_n = (\log n/n)^{\frac{\alpha-k}{2\alpha+1}}$ under the $L_2$ norm.
\end{theorem}

The restriction $\alpha > 3/2$ in the theorem arises from technical requirements to establish posterior contraction for derivatives of order $k$. When applying the general theory to the specific case of using kernel $K_\alpha$, we find that $\alpha > k + 3/2$ is necessary instead of $\alpha > k + 1/2$, which implies $\alpha > 3/2$ for $k \geq 0$.

\begin{remark}
Given the smoothness level of the regression function, the rate-optimal estimation of $f_0$ and its derivatives are achieved under the same choice of the regularization parameter $\lambda$ that does not depend on the derivative order (Theorem~\ref{thm:exp.contraction.deriv} and Theorem~\ref{thm:contraction.deriv}). Therefore, the plug-in GP prior achieves optimal posterior contraction rates simultaneously for all derivative orders, leading to easy parameter tuning. This nonparametric plug-in property in the context of derivative estimation is not seen in some popular nonparametric methods, including smoothing splines and local polynomial regression; see \cite{liu2020non} for more discussion. In the Bayesian paradigm, \cite{yoo2016supremum} established this property for B-splines, and our work reassuringly suggests that GP priors attain the same remarkable property.
We compare the finite-sample performance of these two methods in Section~\ref{sec:simulation}. 
\end{remark}
The special examples in this section focus on covariance kernels and function classes generated by the Fourier basis system, whose practical relevance was discussed earlier. Our general results, including Theorem~\ref{thm:deriv.contraction.l2} and Lemma~\ref{lem:approx.error}, continue to hold for other kernels and function spaces. However, establishing minimax optimality in those settings often requires more delicate analysis and careful selection of GP kernels, even to ensure convergence of point estimators. Future work may extend our results to broader classes of kernels by incorporating tailored rate calculations for key quantities such as $\tilde{\kappa}_{m}$ in Equation~\ref{eq:high.order.kappa} (see Section~\ref{sec:non-asymptotic}), or by leveraging Fourier-analytic techniques as in~\cite{van2008rates}. This represents an interesting direction for further investigation, and some results for additional function spaces will be reported elsewhere. In addition, it would be of interest to extend the results in this section to the $L^\infty$ norm, which is also covered by Theorem~\ref{thm:deriv.contraction.l2}. The main challenge lies in establishing the required convergence rate for the plug-in posterior mean estimator, as in Assumption (D), which involves more intricate analysis due to the nature of the $L^\infty$ norm, particularly when derivatives are involved.

\subsection{Optimal and adaptive hyperparameters tuning}\label{sec:MMLE}

In this section, we study a data-driven empirical Bayes approach for choosing the two hyperparameters: the error variance $\sigma^2$ and the regularization parameter $\lambda$. We will show that our posterior contraction results continue to hold with the estimated $\sigma$. For the data-dependent selection of $\lambda$, we will show that it meets the optimal rate condition for all derivative orders, and adapts to the smoothness level of the true function with an oversmooth kernel.

We first discuss the estimation of the error variance $\sigma^2$. Let $\sigma^2_0$ be its true value. \cite{van2009adaptive} proposed a fully Bayesian scheme by endowing the standard error $\sigma$ with a hyperprior, which is supported on a compact interval $[a,b]\subset (0,\infty)$ that contains $\sigma_0$ with a Lebesgue density bounded away from zero. This approach has been followed by many others \citep{bhattacharya2014anisotropic,li2020comparing}. \cite{de2013semiparametric} showed a Bernstein-von Mises theorem for the marginal posterior of $\sigma$, where the prior for $\sigma$ is relaxed to be supported on $(0,\infty)$. In other words, it is possible to simultaneously estimate the regression function at an optimal nonparametric rate and the standard deviation at a parametric rate.

Here we consider an empirical Bayes approach, which is widely used in practice and eliminates the need for jointly sampling $f$ and $\sigma$. Under model~\eqref{eq:model}, the marginal likelihood is
\begin{equation}\label{eq:marginal}
Y\mid X\sim N(0, \sigma^2(n\lambda)^{-1}K(X,X)+\sigma^2I_n).
\end{equation}
We estimate $\sigma^2$ by its maximum marginal likelihood estimator (MMLE)
\begin{equation}\label{eq:EB.sigma}
\hat{\sigma}^2_n=\lambda Y^T[K(X,X)+n\lambda I_n]^{-1}Y.
\end{equation}
We introduce a modified prior, denoted by $f \sim \GP(0, \hat{\sigma}^2_n (n \lambda)^{-1} K)$ by substituting $\sigma^2$ in the original GP prior with $\hat{\sigma}^2_n$. The Empirical Bayes approach involves replacing $\sigma^2$ with its MMLE $\hat\sigma^2_n$ within the conditional posterior $\Pi_{n,k} (\cdot\mid\data)|_{\sigma^2 = \hat \sigma^2_n}$.

The next theorem shows that the established theory holds under the empirical Bayes scheme. We consider a series representation of $f_0$ using the Fourier basis $\phi_i$ by $f_0 = \s f_i\phi_i$. Let $u_1\geq u_2\geq \ldots\geq u_n$ denote the eigenvalues of the kernel matrix $K(X,X)$.

\begin{theorem}\label{thm:eb.contraction}
Suppose $\s f_i^2/\mu_i^{2r}<\infty$ for some $0<r\leq 1/2$, and $\lambda$ is chosen such that $n^{-1} \sum_{i=1}^{n}\frac{u_i}{n\lambda+u_i} =o(1)$ in $\PP_0^{(n)}$-probability and $\|f_\lambda - f_0\|_\infty = o(1)$. Then, $\hat\sigma_n^2 \rightarrow \sigma_0^2$ in $\PP_0^{(n)}$-probability. Furthermore, Theorem~\ref{thm:deriv.contraction.l2} holds under the empirical Bayes scheme, i.e., $\epsilon_n$ is a posterior contraction rate of $\Pi_{n,k} (\cdot\mid\data)|_{\sigma^2 = \hat \sigma^2_n}$ at $f_0^{(k)}$ under the same conditions therein.
\end{theorem}

\begin{remark}
The parameter $r$ can be understood as a smoothness parameter of $f_0$. When $r=1/2$, the condition $\s f_i^2/\mu_i<\infty$ is equivalent to that $f_0$ belongs to the reproducing kernel Hilbert space $\bbH$ generated by $K$. Hence, by allowing $r \in (0, 1/2)$, Theorem~\ref{thm:eb.contraction} is applicable for $f_0$ that is not in $\bbH$. The other two conditions on $\lambda$ are mild. Note that $n^{-1} \sum_{i=1}^{n}\frac{u_i}{n\lambda+u_i} \leq n^{-1} \sum_{i=1}^{n} \frac{u_i}{n\lambda} \leq \kappa^2 (n\lambda)^{-1}$. The second condition holds if $n\lambda \rightarrow \infty$. All these conditions will be verified in the special examples considered in the next section.
\end{remark}

By verifying the conditions in Theorem~\ref{thm:eb.contraction}, the established results of posterior contraction in the above two examples hold under the empirical Bayes scheme.

\begin{corollary}\label{thm:eb.example}
Theorem~\ref{thm:exp.contraction.deriv} and Theorem~\ref{thm:contraction.deriv} hold for the posterior measure $\Pi_{n,k} (\cdot\mid\data)|_{\sigma^2 = \hat \sigma^2_n}$ with the regularization parameter $\lambda \asymp \log n / n$ and $\lambda\asymp ({\log n}/{n})^{\frac{2\alpha}{2\alpha+1}}$, respectively.
\end{corollary}

We next turn to the regularization parameter $\lambda$ and again consider an empirical Bayes approach. In particular, we choose $\lambda$ by maximizing the marginal likelihood~\eqref{eq:marginal} with $\hat\sigma_n^2$ plugged in for $\sigma^2$. Substituting $\hat\sigma_n^2$ for $\sigma^2$, the marginal likelihood function of $\lambda$ becomes 
\begin{equation}
Y\mid X, \sigma^2 = \hat\sigma_n^2 \sim N(0,n^{-1} Y^T[K(X,X)+n\lambda I_n]^{-1}Y\cdot[K(X,X)+n\lambda I_n]),
\end{equation}
and its logarithm is denoted as $\ell (\lambda \mid X, Y, \sigma^2 = \hat\sigma_n^2)$. Then the MMLE of $\lambda$ is given by
\begin{equation}
\hat \lambda_n = \underset{\lambda > 0}{\text{argmax}} \
 \ell(\lambda \mid X, Y, \sigma^2 = \hat\sigma_n^2).
\end{equation}

The large-sample behavior of $\hat\lambda_n$ can be unstable unless some notion of \textit{exact smoothness} is imposed; see, for example,
\cite{szabo2013empirical}. For technical simplicity, and since
$W^{\alpha_0}$ is a larger class than $H^{\alpha_0}$, we focus on the Sobolev space $W^{\alpha_0}$ in what follows. We introduce refined subclasses $\widetilde W^{\alpha_0}$ and
$\widetilde A^{\gamma_0}$ of $W^{\alpha_0}$ and $A^{\gamma_0}$, respectively.
These subclasses consist of functions
$f_0=\sum_{i\ge1} f_i\psi_i$ that satisfy the usual coefficient conditions of the
ambient space and, in addition, obey a lower bound on their tail coefficients: for every $\delta>0$, there exist constants $C_\delta>0$ and
$N_\delta<\infty$ such that
\begin{equation}\label{eq:lower.bound.condition}
f_i^2 \ge C_\delta\, R(i,\delta), \qquad \text{for all } i \ge N_\delta,
\end{equation}
where $R(i,\delta)= i^{-2\alpha_0-1-\delta}$ for $\widetilde W^{\alpha_0}$ and
$R(i,\delta)= \exp\{-2(\gamma_0+\delta)i\}$ for $\widetilde A^{\gamma_0}$.

The lower bound condition in \eqref{eq:lower.bound.condition} imposes additional regularity in the tail and, in particular, ensures exact smoothness at levels indexed by $\alpha_0$ and $\gamma_0$ by excluding functions that are strictly smoother. Related assumptions
have appeared in several forms in the empirical Bayes and adaptive inference
literature, including exact coefficient decay \citep{szabo2013empirical}, the
polished tail condition \citep{szabo2015frequentist}, and related variants
\citep{gine2010confidence,bull2012honest}. Our proofs indicate that these
alternative conditions could likewise be used to establish the convergence rate of
$\hat\lambda_n$.

The following result provides the rates of $\hat\lambda_n$ for the examples considered in Section~\ref{sec:Matern}, which coincide with the theoretical rate derived in Theorem~\ref{thm:exp.contraction.deriv} and Theorem~\ref{thm:contraction.deriv} up to logarithmic terms. Notably, an oversmooth kernel is sufficient to attain these rates.

\begin{theorem}\label{thm:lambda.mmle}
\begin{enumerate}[(a)]
\item Suppose $f_0\in \widetilde{A}^{\gamma_0}[0,1]$ and the kernel is chosen to be $K_\gamma$ for $\gamma \geq \gamma_0$. Then it holds with $\PP_0^{(n)}$-probability at least $1-n^{-10}$ that $\hat\lambda_n$ is of the order $n^{-\frac{\gamma}{\gamma_0}}$ up to a logarithmic factor.
\item Suppose $f_0 \in \widetilde{W}^{\alpha_0}[0,1]$ and the kernel is chosen to be $K_\alpha$ for $\alpha \geq \alpha_0$ and $\alpha > \frac{2\alpha_0+1}{4\alpha_0-2}$. Then it holds with $\PP_0^{(n)}$-probability at least $1-n^{-10}$ that $\hat \lambda_n$ is of the order $n^{-\frac{2\alpha}{2\alpha_0+1}}$ up to a logarithmic factor.
\end{enumerate}
\end{theorem}

It is easy to see that the rates of $\hat\lambda_n$ satisfy Assumption (C). By verifying Assumption (D) using Lemma~\ref{lem:approx.error} and applying Theorem~\ref{thm:deriv.contraction.l2}, the following two theorems show that the posterior distributions with such a data-driven choice of $\sigma^2$ and $\lambda$, denoted by $\Pi_{n,k} (\cdot\mid\data)|_{\sigma^2 = \hat \sigma^2_n, \lambda = \hat\lambda_n}$, achieve the same convergence rates as obtained in Theorem~\ref{thm:exp.contraction.deriv} and Theorem~\ref{thm:contraction.deriv}. This establishes the nonparametric plug-in property for the plug-in GP method under the empirical Bayes scheme.

\begin{theorem}\label{thm:adaptive.exp}
Suppose $f_0\in \widetilde{A}^{\gamma_0}[0,1]$. If $K_\gamma$ is used in the GP prior with $\gamma\geq\gamma_0>1/2$ and $\lambda$ is estimated by the MMLE $\hat\lambda_n$, then for any $k\in\mathbb{N}_0$, the posterior distribution $\Pi_{n,k} (\cdot\mid\data)|_{\sigma^2 = \hat \sigma^2_n, \lambda = \hat\lambda_n}$ contracts at $f_0^{(k)}$ at a nearly parametric rate $\epsilon_n = 1/\sqrt{n}$ up to logarithmic factors under the $L_2$ norm.
\end{theorem}

\begin{theorem}\label{thm:adaptive.poly}
Suppose $f_0 \in \widetilde{W}^{\alpha_0}[0,1]$. If $K_\alpha$ is used in the GP prior with $\alpha \geq \alpha_0 > k+1/2$, $\alpha > \frac{2\alpha_0+1}{4\alpha_0-2}$ and $\lambda$ is estimated by the MMLE $\hat\lambda_n$, then for any $k\in\mathbb{N}_0$, the posterior distribution $\Pi_{n,k} (\cdot\mid\data)|_{\sigma^2 = \hat \sigma^2_n, \lambda = \hat\lambda_n}$ contracts at $f_0^{(k)}$ at a nearly minimax optimal rate $\epsilon_n = n^{-\frac{\alpha_0-k}{2\alpha_0+1}}$ up to logarithmic factors under the $L_2$ norm.
\end{theorem}

Therefore, we have shown that empirical Bayes remarkably adapts to the unknown smoothness level of the function and achieves optimal estimation for derivatives of any order. This is reminiscent of earlier literature on adaptive empirical Bayes; for example, \cite{castillo2018empirical} and \cite{castillo2020spike} illustrate how empirical Bayes estimation of hyperparameters yields adaptive posterior contraction rates in the sparse normal means model with a spike and slab prior. Similarly to the preceding discussion on $\sigma$, one may alternatively consider a fully Bayesian approach by placing a prior on $\lambda$ as in \cite{van2009adaptive} that has been followed by others such as \cite{ bhattacharya2014anisotropic, Li+Ghosal:17}, although our hyperparameter $\lambda$ has a slightly different meaning. It has been shown that, under conditions, the fully Bayes posterior achieves the same contraction rate as the empirical Bayes posterior in the white noise model \citep{szabo2013empirical} and other nonparametric models \citep{rousseau2017asymptotic}.

\section{Non-asymptotic analysis of key quantities}
\label{sec:non-asymptotic}
In this section, we present the error rates for a list of key quantities in GP regression that are useful for deriving our nonparametric plug-in theory and verifying theorem assumptions.

We first study the rates for the posterior variances of derivatives of Gaussian processes, i.e., the variance of $\Pi_{n,k}(\cdot\mid\data)$; this error bound will be used in establishing Theorem~\ref{thm:deriv.contraction.l2}. Here we consider a general kernel $K\in C^{2k}([0,1],[0,1])$. The posterior covariance of the $k$-th derivative of $f$, denoted by $\tilde{V}^k_n(x, x')$, is given in Equation~\eqref{eq:deriv.variance}. Comparing Equations~\eqref{eq:variance} and~\eqref{eq:deriv.variance}, we can rewrite $\tilde{V}^k_n(x, x')$ as 
$
\tilde{V}^k_n(x, x') = \partial^k_{x}\partial^k_{x'}\hat{V}_n(x,x').
$
Therefore, the posterior covariance of the derivative of GP is exactly the mixed derivative of the posterior covariance of the original GP. This is expected as the differential operator is linear. We write $\tilde{V}^{k}_n(x)=\tilde{V}^k_n(x,x)$ for the posterior variance. We caution that, however, $\tilde{V}^{k}_n(x)$ may not be obtained by taking the derivatives of $\hat{V}_n(x)$ since differentiation and evaluation at $x=x'$ may not be exchangeable. The next lemma provides a non-asymptotic error bound for $\tilde{V}^{k}_n(x)$.

\begin{lemma}\label{thm:equivalent.sigma.deriv}
Let $K\in C^{2k}([0,1],[0,1])$ and suppose Assumption (A) holds. If $\lambda$ is chosen such that $\tilde{\kappa}^2=o(\sqrt{n/\log n})$, then with $\PP_0^{(n)}$-probability at least $1 - n^{-10}$, 
$\|\tilde{V}^{k}_{n}\|_\infty \leq {2\sigma^2\tilde{\kappa}_{k}^2}/{n}.$
\end{lemma}

Next, we present the rates for the higher-order analog of the effective dimension under specific examples, which help verify Assumption (C). We consider $K_\gamma$ and $K_\alpha$ defined in Section~\ref{sec:Matern} for the GP prior and a uniform sampling process. We denote the equivalent kernel of $K_\gamma$ and $K_\alpha$ by $\tilde K_\gamma$ and $\tilde K_\alpha$, respectively.

Let the higher-order analog of the effective dimension for $K_\gamma$ and $K_\alpha$ be $\tilde{\kappa}_{\gamma,m}^2$ and $\tilde{\kappa}_{\alpha,m}^2$, respectively, where $m\in\mathbb{N}_0$ and the subscripts $\gamma$ and $\alpha$ emphasize the use of specific kernels compared to the general definition in Equation~\eqref{eq:high.order.kappa}. Note that we allow $m=0$, which corresponds to $\tilde{\kappa}_{\gamma,0}^2=\tilde{\kappa}_{\gamma}^2$ and $\tilde{\kappa}_{\alpha,0}^2=\tilde{\kappa}_{\alpha}^2$. The same convention in notation applies to $\hat{\kappa}_{\gamma,m}^2$ and $\hat{\kappa}_{\alpha,m}^2$. Lemma~\ref{lem:differentiability.exp} and Lemma~\ref{lem:differentiability.matern} provide the exact order of $\tilde{\kappa}_{\gamma,m}^2$, $\hat{\kappa}_{\gamma,m}^2$ and $\tilde{\kappa}_{\alpha,m}^2$, $\hat{\kappa}_{\alpha,m}^2$ with respect to the regularization parameter $\lambda$.

\begin{lemma}\label{lem:differentiability.exp}
$\tilde K_\gamma\in C^{2m}([0,1]\times[0,1])$ and $\tilde{\kappa}_{\gamma,m}^2 \asymp \hat{\kappa}_{\gamma,m}^2 \asymp (-\log \lambda)^{2m+1}$ for any $\gamma>0$, $0<\lambda<1$ and $m\in\mathbb{N}_0$.
\end{lemma}

\begin{lemma}\label{lem:differentiability.matern}
$\tilde K_\alpha\in C^{2m}([0,1]\times[0,1])$ and $\tilde{\kappa}_{\alpha,m}^2 \asymp \hat{\kappa}_{\alpha,m}^2 \asymp \lambda^{-\frac{2m+1}{2\alpha}}$ for any $\alpha>m+1/2$ and $m\in\mathbb{N}_0$.
\end{lemma}

We also derive non-asymptotic convergence rates of $\hat{f}_n^{(k)}$ for analytic-type functions and periodic Sobolev and H\"older functions considered in Section~\ref{sec:Matern}. These results facilitate justifying Assumption (D).

\begin{lemma}\label{thm:minimax.diff.exp}
Suppose $f_0\in A^\gamma[0,1]$ for $\gamma>1/2$. If $K_\gamma$ is used in the GP prior, then for any $k\in\mathbb{N}_0$ it holds with $\PP_0^{(n)}$-probability at least $1-n^{-10}$ that
$
\|\hat{f}_n^{(k)}-f_0^{(k)}\|_2\lesssim {(\log n)^{k+1}}/{\sqrt{n}},
$
with the corresponding choice of regularization parameter $\lambda\asymp \log n/n$.
\end{lemma}

\begin{lemma}\label{thm:minimax.diff.matern}
Suppose $f_0 \in W^\alpha[0,1]$ or $f_0\in H^{\alpha}[0,1]$ for $\alpha>k+1/2$ and $k\in\mathbb{N}_0$. If $K_\alpha$ is used in the GP prior, then it holds with $\PP_0^{(n)}$-probability at least $1-n^{-10}$ that
$
\|\hat{f}_n^{(k)}-f_0^{(k)}\|_2\lesssim \left({\log n}/{n}\right)^{\frac{\alpha-k}{2\alpha+1}},
$
with the corresponding choice of regularization parameter $\lambda\asymp ({\log n}/{n})^{\frac{2\alpha}{2\alpha+1}}$.
\end{lemma}

Finally, we present convergence rates of $\hat{f}_n^{(k)}$ using empirical Bayes estimators of $\lambda$ and oversmooth kernels, which verify Assumption (D) for deriving Theorem~\ref{thm:adaptive.exp} and Theorem~\ref{thm:adaptive.poly}. The proof hinges on Lemma~\ref{lem:approx.error}, which is dependent on the approximation function $f_\lambda$.

\begin{lemma}\label{lem:exp.mismatch}
Suppose $f_0\in \widetilde{A}^{\gamma_0}[0,1]$, the kernel is chosen to be $K_\gamma$ for $\gamma \geq \gamma_0 > 1/2$ and $k\in\mathbb{N}_0$, and $\lambda$ is estimated by the MMLE $\hat\lambda_n$. Then it holds with $\PP_0^{(n)}$-probability at least $1-n^{-10}$ that
$
\|\hat f_n^{(k)} - f_0^{(k)}\|_2 \lesssim {(\log n)^{k+1}}/{\sqrt{n}}.
$
\end{lemma}

\begin{lemma}\label{lem:poly.mismatch}
Suppose $f_0 \in \widetilde{W}^{\alpha_0}[0,1]$, the kernel is chosen to be $K_\alpha$ for $\alpha \geq \alpha_0 > k+1/2$, $\alpha > \frac{2\alpha_0+1}{4\alpha_0-2}$ and $k\in\mathbb{N}_0$, and $\lambda$ is estimated by the MMLE $\hat\lambda_n$. Then it holds with $\PP_0^{(n)}$-probability at least $1-n^{-10}$ that
$
\|\hat f_n^{(k)} - f_0^{(k)}\|_2 \lesssim n^{-\frac{\alpha_0-k}{2\alpha_0+1}} \sqrt{\log n}.
$
\end{lemma}

\section{Simulation}\label{sec:simulation}

We conduct a simulation study to assess the finite sample performance of the proposed plug-in procedure for function derivatives. The true function is $f_0(x)=\sqrt{2}\s i^{-4}\sin i\cos[(i-1/2)\pi x]$, $x\in[0,1]$, which has H\"older smoothness level $\alpha=3$. We simulate $n$ observations from the regression model $Y_i=f_0(X_i)+\varepsilon_i$ with $\varepsilon_i\sim N(0, 0.1)$ and $X_i\sim \text{Unif}[0,1]$. We consider three sample sizes: 100, 500, and 1000, and replicate the simulation 100 times. 

For Gaussian process priors, we consider three commonly used covariance kernels: the Mat\'ern, squared exponential (SE), and second-order Sobolev kernels, which are given by
\begin{align}
K_{\text{Mat},\nu}(x,x')&=\frac{2^{1-\nu}}{\Gamma(\nu)}\left(\sqrt{2\nu}|x-x'|\right)^\nu B_\nu\left(\sqrt{2\nu}|x-x'|\right),\\
K_{\text{SE}}(x,x')&=\exp(-(x-x')^2),\\
K_{\text{Sob}}(x,x')&=1+xx'+\min\{x,x'\}^2(3\max\{x, x'\}-\min\{x, x'\})/6.
\end{align}
Here $B_\nu(\cdot)$ is the modified Bessel function of the second kind with $\nu$ being the smoothness parameter to be determined. For the Mat\'ern kernel, it is well known that the eigenvalues of $K_{\text{Mat},\nu}$ decay at a polynomial rate, that is,  $\mu_i\asymp i^{-2(\nu+1/2)}$ for $i\in\mathbb{N}$. 

We compare various Gaussian process priors with a random series prior using B-splines. B-splines are widely used in nonparametric regression \citep{james2009functional,wang2020functional}. In the context of estimating function derivatives, the B-spline prior with normal basis coefficients has been recently shown to enjoy the plug-in property (cf. Theorem 4.2 in \cite{yoo2016supremum}), which, to the best of our knowledge, constitutes the only Bayesian method in the existing literature with this property. The implementation of this B-spline prior follows \cite{yoo2016supremum}. In particular, for any $x\in[0,1]$, let $ b_{J,4}(x)=(B_{j,4}(x))_{j=1}^J$ be a B-spline of order $4$ and degrees of freedom $J$ with uniform knots. The prior on $f$ is given as $f(x)=b_{J,4}(x)^T \beta$ with each entry of $\beta$ following $N(0, \sigma^2)$ independently. The unknown variance $\sigma^2$ is estimated by its MMLE $\hat{\sigma}_n^2=n^{-1}Y^T(B B^T+I_n)^{-1}Y$, where $B=(b_{J,4}(X_1),\ldots,b_{J,4}(X_n))^T$. The number of interior knots $N=J-4$ is selected from $\{1, 2, \ldots, 10\}$ using leave-one-out cross-validation. 

For Gaussian process priors, we adopt the same strategy of leave-one-out cross-validation to select the degree of freedom $\nu$ in $K_{\text{Mat},\nu}$ for a fair comparison. In particular, we consider $\nu$ from the set \{3, 3.5, 4, 4.5, \ldots, 10\}. The regularization parameter $\lambda$ and unknown $\sigma^2$ are determined by empirical Bayes through maximizing the marginal likelihood. In addition to the estimation using each kernel, we also automatically select the best kernel among the three via leave-one-out cross-validation.

For each method, we evaluate the posterior mean $\hat{f}_n$ and $\hat{f}_n'$ at 100 equally spaced points in $[0,1]$, and calculate the root mean square error (RMSE) between the estimates and the true functions:
$
\mathrm{RMSE} = \sqrt{\sum_{t = 0}^{99} \{\hat{s}(t/99) - s(t /99)\}^2/100}, 
$
where $\hat{s}$ is the estimated function ($\hat{f}_n$ or $\hat{f}_n'$) and $s$ is the true function ($f_0$ or $f_0'$). 

Table~\ref{table:RMSE} reports the average RMSE of all methods over 100 repetitions for $f_0$ and $f_0'$. Clearly, the RMSE of all methods steadily decreases as the sample size $n$ increases. The squared exponential kernel and Mat\'ern kernel are the two leading approaches for all sample sizes and for both $f_0$ and $f_0'$. While the difference between various methods for $f_0$ tends to vanish when the sample size increases to $n = 1000$, the performance gap in estimating $f_0'$ is more profound. In particular, compared to the squared exponential kernel, the B-spline method increases the average RMSE more than twofold from 0.13 to 0.33. We notice considerably large RMSEs for B-splines in a proportion of simulations, so we also calculate the median RMSEs for better robustness, which are 0.075, 0.032, 0.024, 0.51, 0.31, and 0.24 for the six scenarios. This improves the summarized RMSEs to close to the Sobolev kernel for both $f_0$ and $f_0'$ when $n = 1000$. RMSEs of GP priors with the selected kernel via cross-validation do not significantly deviate from the best-performing kernel relative to the standard errors, suggesting that it can be a useful strategy to choose kernels if desired.

\setlength{\tabcolsep}{4pt}
\begin{table}
\centering
\caption{RMSE of estimating $f_0$ and $f_0'$, averaged over 100 repetitions. The first four rows are the plug-in GP prior with various kernels (Mat\'ern kernel, squared exponential kernel, second-order Sobolev kernel, and the selected kernel via cross-validation). The last row is the random series prior using B-splines. Standard deviations are provided in parentheses.
\vspace{0.2in}
\label{table:RMSE}} 
\resizebox{\textwidth}{!}{
\begin{tabular}{ccccccc}
\toprule
\multicolumn{1}{l}{}            & \multicolumn{3}{c}{$f_0$}                    & \multicolumn{3}{c}{$f_0'$}              \\ \cmidrule(lr){2-4} \cmidrule(lr){5-7}
& $n = 100$           & 500           & 1000          & $n = 100$        & 500         & 1000        \\ \cmidrule{1-1} \cmidrule(lr){2-4} \cmidrule(lr){5-7}
Mat\'ern          & 0.048 (0.020) & 0.024 (0.0075) & 0.019 (0.0060) & 0.24 (0.080) & 0.17 (0.034) & 0.16 (0.030) \\
SE                               & 0.048 (0.020) & 0.022 (0.0075) & 0.017 (0.0057) & 0.22 (0.087) & 0.14 (0.031) & 0.13 (0.024) \\
Sobolev                          & 0.050 (0.018) & 0.031 (0.0061) & 0.028 (0.0049) & 0.31 (0.051) & 0.28 (0.019) & 0.27 (0.016) \\
CV                          & 0.050 (0.018) & 0.025 (0.0079) & 0.019 (0.0062) & 0.28 (0.084) & 0.17 (0.056) & 0.15 (0.042) \\
B-splines          & 0.075 (0.029) &   0.034 (0.013)  &   0.024 (0.0085)         & 0.57 (0.36) &   0.49 (0.58) &  0.33 (0.28)        \\
\bottomrule
\end{tabular}
}
\end{table}

We next choose two representative examples from the 100 repetitions to visualize the estimates of $f_0$ and $f_0'$ in Figure~\ref{fig}. The dotted line stands for the posterior mean $\hat{f}_n$ and dashed lines for the 95\% simultaneous $L_\infty$ credible bands, where the radius is estimated by the 95\% quantile of the posterior samples of $\|f-\hat f_n\|_\infty$ and $\|f'-\hat f_n'\|_\infty$. For the B-spline prior, we use the default setting as in \cite{yoo2016supremum} by specifying the inflation factor $\rho=0.5$. Note that these credible bands have fixed widths by construction.

\newcommand{\scale}{0.22} 
\begin{figure}[ht]
\centering 
\begin{tabular}{c c c @{\hskip0.01pt} c @{\hskip0.01pt} c @{\hskip0.01pt} c}
& &\ \ Mat\'ern
& \ SE & \ Sobolev & \ B-splines \\ 
\begin{sideways}
\rule[0pt]{-0.3in}{0pt} Example 1
\end{sideways} \quad &

\begin{sideways}
\rule[0pt]{0.45in}{0pt} $f_0$
\end{sideways} &

{\includegraphics[width = \scale\linewidth]{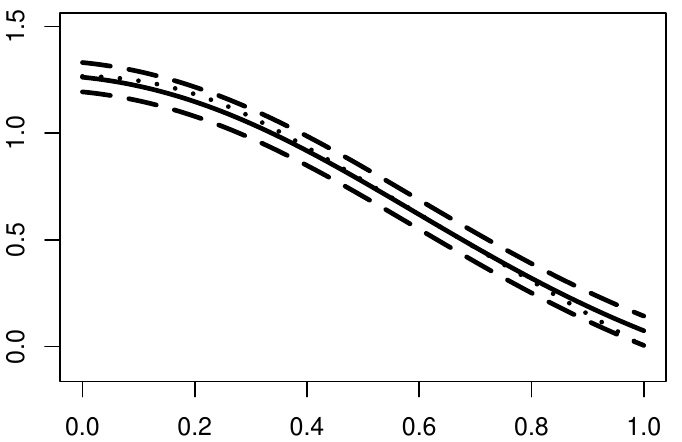}}  & 
{\includegraphics[width = \scale\linewidth]{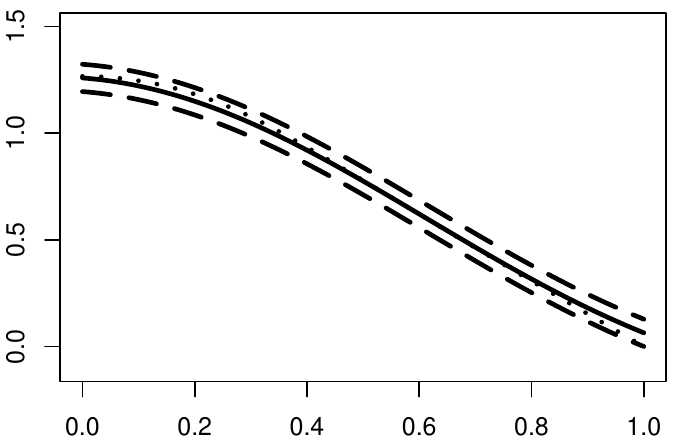}} &
{\includegraphics[width = \scale\linewidth]{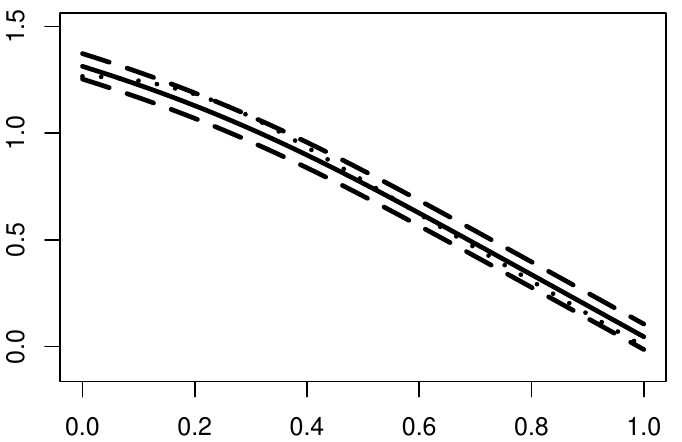}} &
{\includegraphics[width = \scale\linewidth]{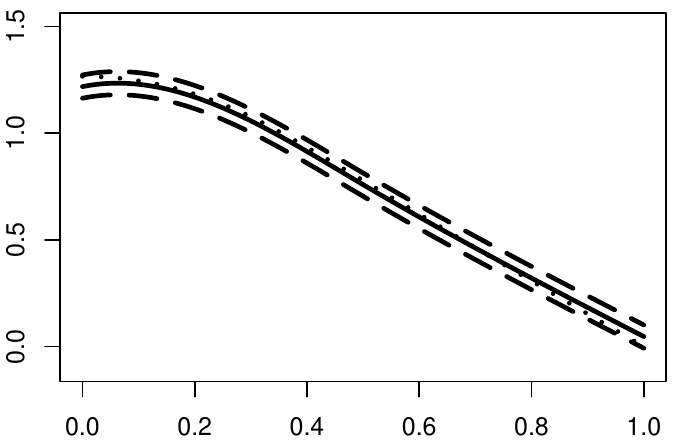}} \\

\vspace{0.2in}

&
\begin{sideways}
\rule[0pt]{0.45in}{0pt} $f_0'$
\end{sideways} &
\includegraphics[width = \scale\linewidth]{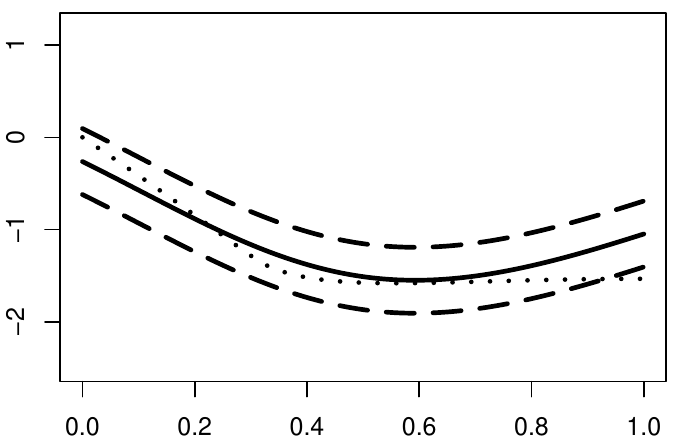} &
\includegraphics[width = \scale\linewidth]{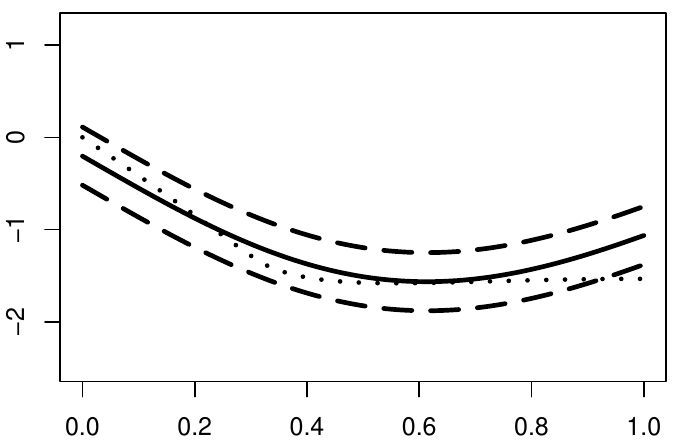} &
\includegraphics[width = \scale\linewidth]{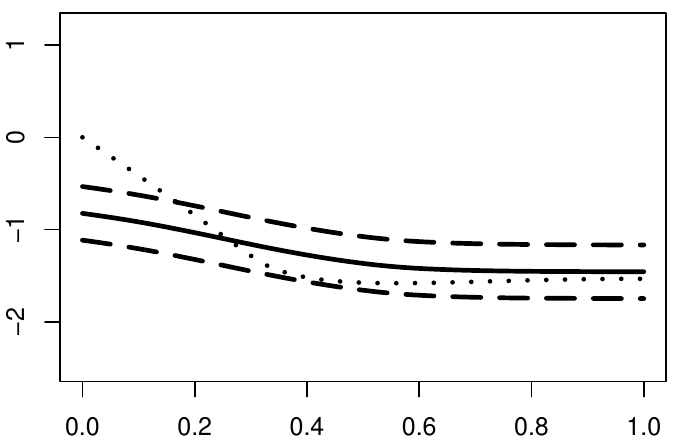} &
\includegraphics[width = \scale\linewidth]{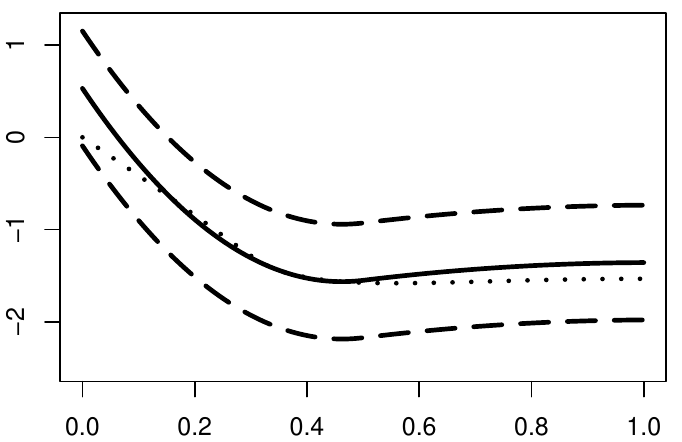} \\

\begin{sideways}
\rule[0pt]{-0.3in}{0pt} Example 2
\end{sideways} &
\begin{sideways}
\rule[0pt]{0.45in}{0pt} $f_0$
\end{sideways} &
\includegraphics[width = \scale\linewidth]{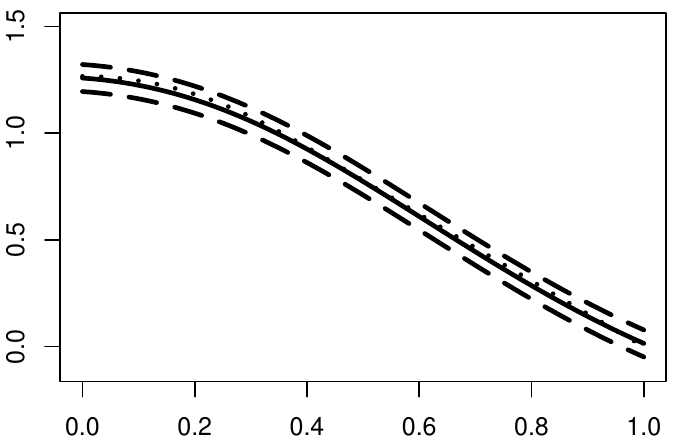} &
\includegraphics[width = \scale\linewidth]{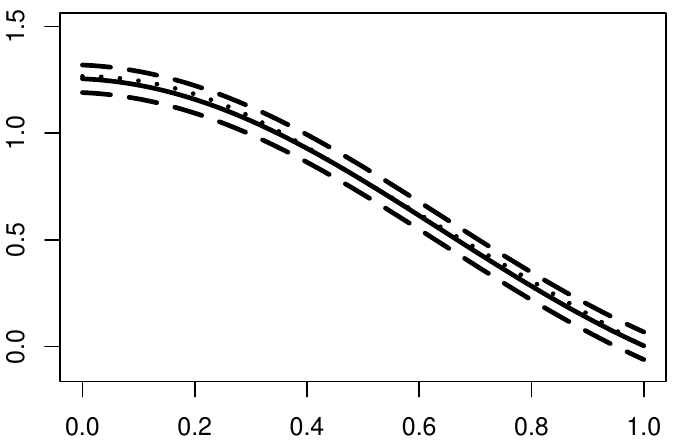} &
\includegraphics[width = \scale\linewidth]{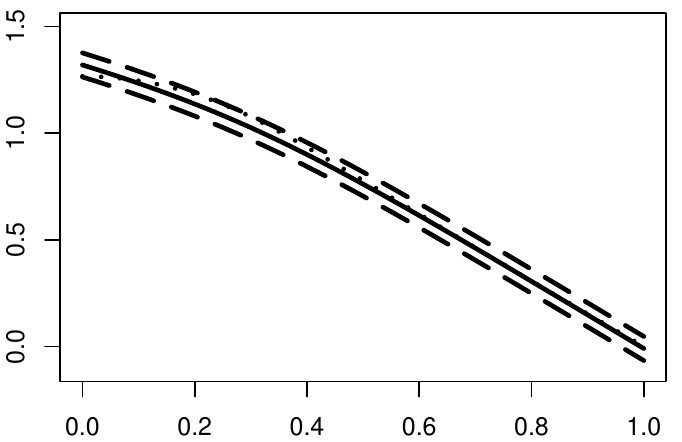} &
\includegraphics[width = \scale\linewidth]{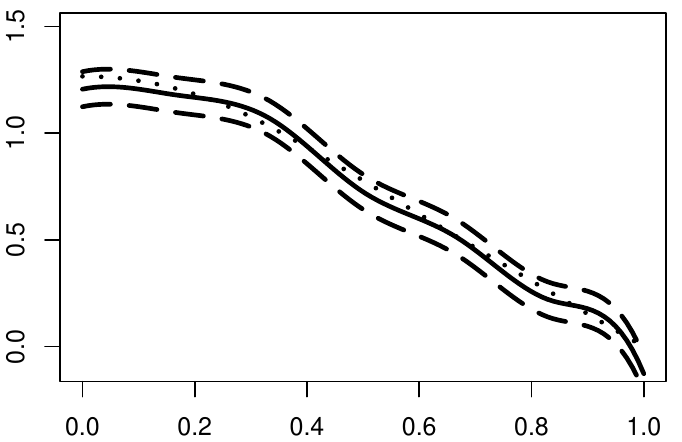} \\

&
\begin{sideways}
\rule[0pt]{0.45in}{0pt} $f_0'$
\end{sideways} &
\includegraphics[width = \scale\linewidth]{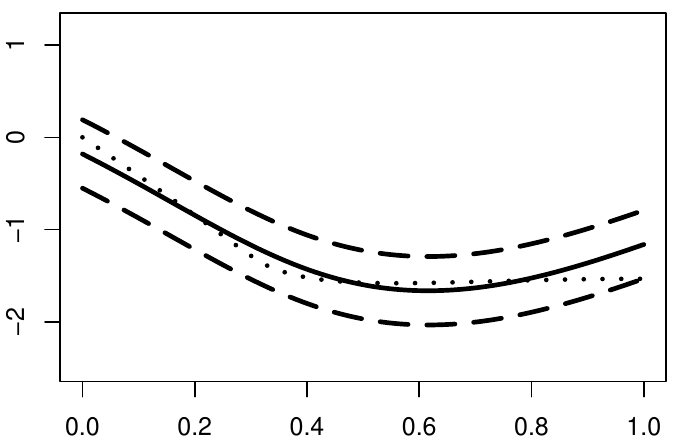} &
\includegraphics[width = \scale\linewidth]{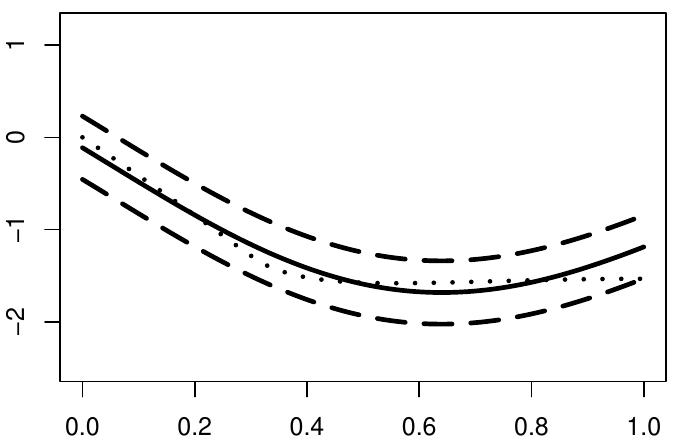} &
\includegraphics[width = \scale\linewidth]{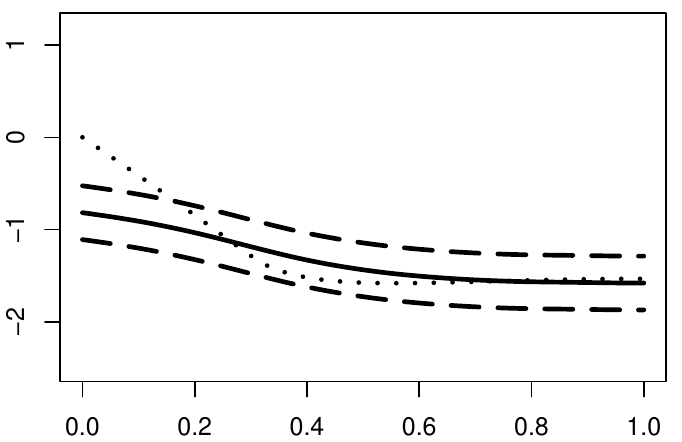} &
\includegraphics[width = \scale\linewidth]{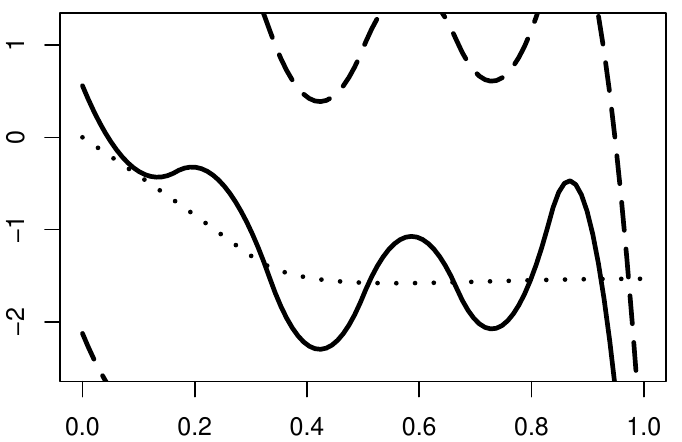} \\
\end{tabular} 
\caption{Visualization of estimates of $f_0$ and $f_0'$ in two examples with $n=1000$. Solid: posterior mean $\hat{f}_n$; Dots: true function $f_0$ or $f_0'$; Dashes: 95\% $L_\infty$-credible bands.}
\label{fig}
\end{figure}

The first and third rows of Figure~\ref{fig} show that the four methods lead to comparable estimation and uncertainty quantification when estimating $f_0$. However, the deviation between different methods is considerably widened for the estimation of $f_0'$. Mat\'{e}rn and squared exponential kernels constantly give the most accurate point estimation, and their credible bands cover the ground truth with reasonable width, indicating the effectiveness of the nonparametric plug-in procedure using GP priors. There is a tendency for the second-order Sobolev kernel to fail to capture $f_0'$ around the left endpoint. The performance of the B-spline method continues to exhibit sensitivity to the choice of $N$, selected by leave-one-out cross-validation. In Example 1, the selected $N$ is 1, and the B-spline prior yields comparable credible bands of $f_0'$ than GP priors with slightly altered estimation near zero; in Example 2 shown in the fourth row, the B-spline method with $N=5$ gives a point estimate that is substantially worse than the other three GP methods, and the associated credible bands are off the chart. We remark that the performance of B-splines might have been substantially improved had the number of knots been selected by a different tuning method or with a different simulation setting. While leave-one-out cross-validation may be appropriate for selecting $N$ when estimating $f_0$, as observed in~\cite{yoo2016supremum}, our results suggest that adjustments or alternative strategies seem to be needed when the objective is to make inference on $f_0'$. For GP priors, our numerical results suggest choosing $\lambda$ using empirical Bayes seems to be a reasonable strategy for both $f_0$ and $f_0'$, which is in line with the nonparametric plug-in property of GP priors.

In another simulation reported in Appendix~\ref{sec:simulation1}, we compare our plug-in GP estimators with the inverse method proposed in \cite{holsclaw2013gaussian}, where the authors perceived the plug-in estimator as suboptimal. We consider the regression function $f_0 (x) = x \sin(x)/10$ given by one simulated example in \cite{holsclaw2013gaussian}; the implementation of the inverse method follows the authors' setup and uses their published code online. We generate $n = 100$ and 500 data points on a regular grid in $[0, 10]$, and generate noise $\varepsilon_i\sim N(0, \sigma^2)$ with $\sigma = 0.1$, 0.2 and 0.3. The goal is to estimate $f_0'$. We have found that the inverse method has a noticeable frequency to produce a zero estimate, and the estimates have extremely high variability across replications, partly due to complicated and poor hyperparameter tuning. We compare its performance with and without these simulations in which the inverse method is not stable. The proposed plug-in GP method continues to give competitive performance, and we have found no scenario in the considered settings with varying sample size and noise standard deviation where the inverse method has significantly smaller RMSE relative to standard errors. In contrast, plug-in GPs are significantly better when the sample size is $n = 500$ and noise standard deviationn $\sigma = 0.1$. We also note that the inverse method is not only computationally intensive but also restrictive to one particular derivative order, and its generalization to other derivative orders is nontrivial. The proposed plug-in GP is instead computationally efficient and applicable for general derivative orders. 

In addition to the empirical Bayes approach to choose $\sigma^2$ and $\lambda$, in Appendix~\ref{sec:simulation1}, we also implement a fully Bayesian alternative where inverse Gamma and Gamma priors are assigned to $\sigma^2$ and $\lambda$, respectively. We do not observe any significant difference between the results of the two treatments under the considered simulation setting.

\section{Real data application}

We apply the proposed plug-in GP method to analyze the rate of global sea rise using global mean sea-level (GMSL) records from coastal and island tide-gauge measurements from the late 19th to early 21st century. This tide-gauge record dataset has been described and previously analyzed in \cite{church200620th,church2011sea}. Our analysis uses three variables: time ($x$) in years AD, GMSL ($y$) in millimeters, and one-sigma sea-level observational error ($\sigma_y$) in millimeters. The sample size is $n = 130$. We are interested in estimating the rate of GMSL rise, denoted by $f'$, with $f$ being the regression function of $y$ on $x$. 

As a descriptive summary, the total GMSL rise from January 1880 to December 2009 is about 210 mm over the 130 years. The least-squares linear trend of sea-level rise from 1900 AD to 2009 AD is 1.7 mm/yr $\pm$ 0.3 mm/yr \citep{church200620th}. Still using linear regression, \cite{church2011sea} reduced the uncertainty and obtained an estimate of 1.7 mm/yr $\pm$ 0.2 mm/yr for the same period. Nonparametric inference on $f'(x)$ as an unknown function allows us to depict flexible rate changes that are possibly time-varying. 

We assume a nonparametric regression model for the observed data $y_i = f(x_i) + \varepsilon_i$, for $i = 1880, 1881, \ldots, 2009$. We consider two error structures: (i) a homogeneous error structure with $\varepsilon_i \sim N(0, \sigma^2)$, and (ii) a heterogeneous structure incorporating the observational error $\sigma^2_{y_i}$, that is, $\varepsilon_i \sim N(0, \sigma^2_{y_i}+\sigma^2)$. Our proposed method can be applied to these two models straightforwardly. We assign a Gaussian process prior $f \sim \GP(0, \sigma^2 (n \lambda)^{-1} K)$, and use leave-one-out cross-validation to choose the kernel function from the Mat\'ern kernel, squared exponential kernel, and Sobolev kernel. For both error structures, the Sobolev kernel is selected. We optimize $\sigma^2$ and $\lambda$ jointly by maximizing the marginal likelihood function, noting that for model (i) the closed-form expression for $\sigma^2$ is available as discussed in Section~\ref{sec:MMLE}. Then the plug-in posterior distribution of $f'$ is analytically tractable and would require no further tuning or posterior sampling, based on which one can make inference on the rate of global sea-level rise.

For both error structures, Figure~\ref{fig:GMSL} presents the rate estimation using the posterior mean function along with 90\% pointwise and simultaneous credible bands for uncertainty quantification. Model (a) reveals that the rate of GMSL rise is not a constant over time, and instead has accelerated much from 1.22 mm/yr in 1880 AD to 1.81 mm/yr in 2010 AD. Model (b) shows a similar acceleration pattern but with a narrower overall range and credible band, suggesting reduced uncertainty by incorporating the observational variance $\sigma_{y_i}^2$ in the model. Our findings of the nonlinear, increasing rate of global sea-level rise have also been observed in earlier literature, such as \cite{cahill2015modeling}, based on integrated Gaussian processes, which necessitate complex parameter tuning and computationally intensive posterior sampling. Our uncertainty quantification indicates that the rate of 1.7 mm/yr in a least-squares line roughly falls into the 90\% simultaneous credible intervals from approximately 1955 AD to 1985 AD. Overall, our analysis captures the acceleration in the rate of global sea rise that would otherwise be missed by least-squares fitting, indicates the least-squares slope of 1.7mm/yr may be only representative for a limited time period, and provides a time-varying description of global sea-level rise for the past 130 years.

\begin{figure}[H]
\centering
\subfigure[$\varepsilon_i\sim N(0,\sigma^2)$.]{
\includegraphics[width=0.475\textwidth]{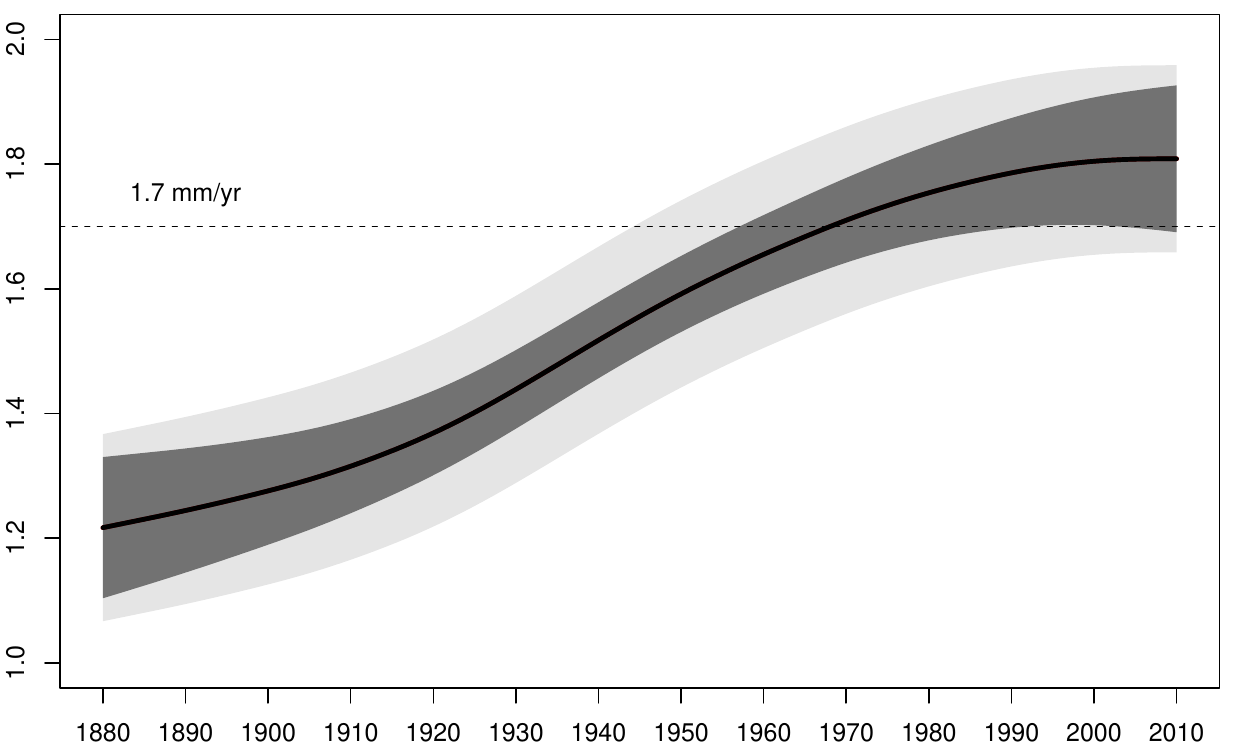}
} 
\subfigure[$\varepsilon_i\sim N(0,\sigma^2_{y_i}+\sigma^2)$.]{
\includegraphics[width=0.475\textwidth]{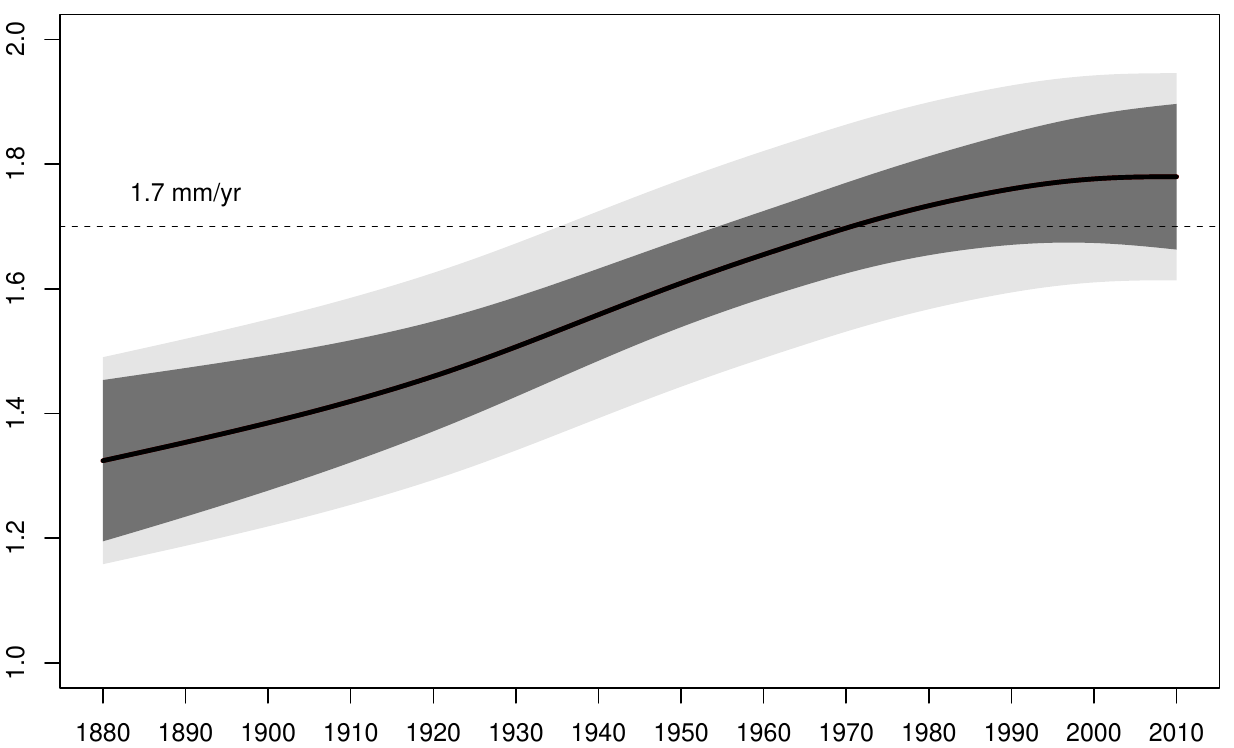}
} 
\caption{Rate of global sea-level rise calculated ($f'$) under each error structure. Shading denotes 90\% pointwise (dark) and simultaneous credible bands (light) for the rate process. The reference rate of 1.7 mm/yr is the least-squares slope in a linear model assuming a constant rate of sea-level rise.}
\label{fig:GMSL}
\end{figure}

\newpage

\begin{appendices}

\section{Proofs}\label{sec:proofs}

In this section, we present the proofs of all theoretical results in the main paper, auxiliary technical results, and their proofs.

\subsection{Proofs in Section~\ref{sec:main.results}}

\begin{proof}[Proof of Theorem~\ref{thm:deriv.contraction.l2}]
Consider the centered Gaussian process: $(f^{(k)}-\hat{f}^{(k)}_n)\mid \data\sim {\rm GP}(0, \tilde{V}_n^k)$. Let $\EE\|f^{(k)}-\hat{f}^{(k)}_n\|_\infty =: \EE(\|f^{(k)}-\hat{f}^{(k)}_n\|_\infty \mid \data)$ denote the conditional expectation with respect to the posterior distribution to ease notation. According to Borell-TIS inequality (cf. Proposition A.2.1 in \cite{van1996weak}), we have for any $j\in \mathbb{N}$,
\begin{equation}
\Pi_{n,k}\left(\left|\|f^{(k)}-\hat{f}^{(k)}_n\|_\infty-\EE\|f^{(k)}-\hat{f}^{(k)}_n\|_\infty\right|>j\epsilon_n\,\big|\,\data\right)\leq 2\exp\left(-j^2\epsilon_n^2/2\|\tilde{V}^{k}_n\|_\infty\right).
\end{equation}
Note that
\begin{align}
& \Pi_{n,k}\left(\left|\|f^{(k)}-\hat{f}^{(k)}_n\|_\infty-\EE\|f^{(k)}-\hat{f}^{(k)}_n\|_\infty\right|>j\epsilon_n\,\big|\,\data\right)\\
& \qquad \qquad \geq \Pi_{n,k}\left(\|f^{(k)}-\hat{f}^{(k)}_n\|_\infty-\EE\|f^{(k)}-\hat{f}^{(k)}_n\|_\infty>j\epsilon_n\,\big|\,\data\right).
\end{align}
Thus, by Lemma~\ref{thm:equivalent.sigma.deriv}, it holds with $\PP_0^{(n)}$-probability at least $1-n^{-10}$ that
\begin{equation}
\|\tilde{V}^{k}_n\|_\infty \leq \frac{2\sigma^2\tilde{\kappa}_{k}^2}{n}, 
\end{equation}
leading to 
\begin{equation}
\Pi_{n,k}\left(\|f^{(k)}-\hat{f}^{(k)}_n\|_\infty-\EE\|f^{(k)}-\hat{f}^{(k)}_n\|_\infty>j\epsilon_n\,\big|\,\data\right)\leq 2\exp\left(-n\epsilon_n^2j^2/4\sigma^2\tilde{\kappa}_{k}^2\right).
\end{equation}
By Lemma~\ref{lem:entropy.integral.deriv}, there exists $C_1>0$ such that with probability at least $1-n^{-10}$,
\begin{equation}\label{eq:entropy.integral.deriv}
\EE\|f^{(k)}-\hat{f}^{(k)}_n\|_\infty\leq C_1 \tilde{\kappa}_{k}\sqrt{\frac{\log n}{n}}.
\end{equation}

We next consider the two cases when $p=\infty$ and $p = 2$ separately. If $p = \infty$, by Assumption (D), with $\PP_0^{(n)}$-probability tending to 1 we have $\|\hat{f}_n^{(k)}-f_0^{(k)}\|_\infty\lesssim \epsilon_n$ and thus
\begin{equation}\label{eq:supremum.eq1}
\|f^{(k)}-\hat{f}_n^{(k)}\|_\infty\geq \|f^{(k)}-f_0^{(k)}\|_\infty-\|\hat{f}_n^{(k)}-f_0^{(k)}\|_\infty\geq \|f^{(k)}-f_0^{(k)}\|_\infty-C\epsilon_n
\end{equation}
for some $C>0$, which implies
\begin{equation}
\Pi_{n,k}\left(\|f^{(k)}-f_0^{(k)}\|_\infty>C\epsilon_n+C_1\tilde{\kappa}_{k}\sqrt{\frac{\log n}{n}}+j\epsilon_n\,\big|\,\data\right)\leq 2\exp\left(-n\epsilon_n^2j^2/4\sigma^2\tilde{\kappa}_{k}^2\right).
\end{equation}
According to Assumption (C), $\tilde{\kappa}_{k}^2=O(n\epsilon_n^2/\log n)$ and $\tilde{\kappa}_{k}\sqrt{\frac{\log n}{n}}\leq C_2\epsilon_n$ for some $C_2>0$, there exists a $J\in\mathbb{N}$ such that $C+C_1C_2\leq J$. Thus, for any $j\geq J$, it holds with $\PP_0^{(n)}$-probability tending to 1 that
\begin{equation}
\Pi_{n,k}\left(\|f^{(k)}-f_0^{(k)}\|_\infty>2j\epsilon_n\,\big|\,\data\right)\leq 2\exp\left(-n\epsilon_n^2j^2/4\sigma^2\tilde{\kappa}_{k}^2\right),
\end{equation}
which implies that
\begin{equation}
\Pi_{n,k}\left(\|f^{(k)}-f_0^{(k)}\|_\infty>j\epsilon_n\,\big|\,\data\right)=\exp\left(-Kn\epsilon_n^2j^2/\tilde{\kappa}_{k}^2\right)
\end{equation}
for some $K>0$. Let $A_n$ denote the event that the preceding display holds, which satisfies $\PP_0^{(n)}(A_n^c)\rightarrow0$. Hence, for any $M_n\rightarrow \infty$,
\begin{align}
& \PP_0^{(n)}\Pi_{n,k}\left(\|f^{(k)}-f_0^{(k)}\|_\infty>M_n\epsilon_n\,\big|\,\data\right)\\
&\qquad \qquad  \leq \PP_0^{(n)}\Pi_{n,k}\left(\|f^{(k)}-f_0^{(k)}\|_\infty>M_n\epsilon_n\,\big|\,\data\right)\mathbbm{1}(A_n)+\PP_0^{(n)}(A_n^c)
\rightarrow 0.
\end{align}
Therefore, $\epsilon_n$ is a contraction rate under the $L_\infty$ norm.

Now we consider the case of $p=2$. Since $\|\hat{f}_n^{(k)}-f_0^{(k)}\|_2\lesssim\epsilon_n$ with $\PP_0^{(n)}$-probability tending to 1, we have
\begin{equation}
\|f^{(k)}-\hat{f}_n^{(k)}\|_\infty\geq \|f^{(k)}-\hat{f}_n^{(k)}\|_2\geq \|f^{(k)}-f_0^{(k)}\|_2-\|\hat{f}_n^{(k)}-f_0^{(k)}\|_2\geq \|f^{(k)}-f_0^{(k)}\|_2-C\epsilon_n.
\end{equation}
Comparing the preceding display with Equation~\eqref{eq:supremum.eq1} and following the same arguments, we have with $\PP_0^{(n)}$-probability tending to 1 that
\begin{equation}
\Pi_{n,k}\left(\|f^{(k)}-f_0^{(k)}\|_2 > j\epsilon_n\,\big|\,\data\right)\leq \exp\left(-Kn\epsilon_n^2j^2/\tilde{\kappa}_{k}^2\right).
\end{equation}
Then a similar argument as in the case of $p = \infty$ yields that $\epsilon_n$ is also a posterior contraction rate. This completes the proof.
\end{proof}

\begin{proof}[Proof of Lemma~\ref{lem:approx.error}]
According to the proof of Theorem 2 in \cite{liu2020non} and substituting $\delta=n^{-10}$, it holds with $\PP_0^{(n)}$-probability at least $1 - n^{-10}$ that
\begin{equation+}\label{eq:KRR.thm2}
\|\hat{f}_n-f_\lambda\|_{\tilde{\bbH}}\leq \frac{\tilde{\kappa}^{-1}C(n,\tilde{\kappa})}{1-C(n,\tilde{\kappa})}\|f_\lambda-f_0\|_\infty+\frac{1}{1-C(n,\tilde{\kappa})}\frac{C_1\tilde{\kappa}\sigma\sqrt{10\log(3n)}}{\sqrt{n}},
\end{equation+}
where $C_1 > 0$ does not depend on $K$ or $n$ and $C(n,\tilde{\kappa})=\frac{\tilde{\kappa}^2 \sqrt{10\log(3n)}}{\sqrt{n}} \left(4 + \frac{4 \tilde{\kappa}\sqrt{10\log(3n)}}{3 \sqrt{n}} \right)$. By choosing $\lambda$ such that $\tilde{\kappa}^2=o(\sqrt{n/\log n})$, we have $C(n,\tilde{\kappa})\leq 1/2$. Consequently,
\begin{equation}
\|\hat{f}_n-f_\lambda\|_{\tilde{\bbH}}\leq \frac{\tilde{\kappa} \sqrt{10\log(3n)}}{\sqrt{n}} \left(4 + \frac{4 \tilde{\kappa}\sqrt{10\log(3n)}}{3 \sqrt{n}} \right)\|f_\lambda-f_0\|_\infty+\frac{2C_1\tilde{\kappa}\sigma\sqrt{10\log(3n)}}{\sqrt{n}}.
\end{equation}
Assuming $\|f_\lambda-f_0\|_\infty = o(1)$, we have
\begin{equation}
\|\hat{f}_n-f_\lambda\|_{\tilde{\bbH}} \lesssim \tilde\kappa \sqrt{\frac{\log n}{n}}.
\end{equation}
Since $\|f^{(k)}\|_2\leq D(k,\lambda) \|f\|_{\tilde{\bbH}}$ for any $f\in\bbH$, we obtain
\begin{equation}
\|\hat{f}_n^{(k)}-f_0^{(k)}\|_2 \lesssim\|f_\lambda^{(k)}-f_0^{(k)}\|_2 + D(k,\lambda) \tilde\kappa \sqrt{\frac{\log n}{n}}.
\end{equation}
\end{proof}

\begin{proof}[Proof of Theorem~\ref{thm:exp.contraction.deriv}]

Note that the Fourier basis $\{\psi_i\}_{i=1}^\infty$ satisfies Assumption (A) with $C_{k,\psi}=\sqrt{2}(2\pi)^k$ and Assumption (B) with $L_{k,\psi}=\sqrt{2}(2\pi)^{k+1}$. According to Lemma~\ref{lem:differentiability.exp}, we have $\tilde{\kappa}_{\gamma,m}^2\asymp (-\log \lambda)^{2m+1}$ for any $m\in\mathbb{N}_0$. Since $\lambda\asymp \log n/n$, we have $\tilde{\kappa}_\gamma^2\asymp (-\log\lambda) \asymp \log (n/\log n)=o(\sqrt{n/\log n})$. Besides, $\tilde{\kappa}_{\gamma,k}^2\asymp (\log (n/\log n) )^{2k+1} = O(n\epsilon_n^2/\log n)$. It also follows that $\hat{\kappa}_{\gamma,k+1}^2 \asymp (-\log \lambda)^{2k+3} \asymp (\log (n/\log n))^{2k+3}=O(n)$ for any $k\in\mathbb{N}_0$. This verifies Assumption (C). Finally, Assumption (D) is given by Lemma~\ref{thm:minimax.diff.exp}, which shows a convergence rate of $\hat{f}_n^{(k)}$ under the $L_2$ norm is $\epsilon_n=(\log n)^{k+1}/\sqrt{n}$ when $\gamma>1/2$. Invoking Theorem~\ref{thm:deriv.contraction.l2}, $\epsilon_n$ is a contraction rate of the posterior distribution $\Pi_{n,k}(\cdot \mid \data)$. This completes the proof. 
\end{proof}

\begin{proof}[Proof of Theorem~\ref{thm:contraction.deriv}]
The arguments are similar to those used to prove Theorem~\ref{thm:exp.contraction.deriv}, and we only note the key differences below. We first verify Assumption (C). According to Lemma~\ref{lem:differentiability.matern}, we have $\tilde{\kappa}_{\alpha,m}^2\asymp \lambda^{-\frac{2m+1}{2\alpha}}$ for any $m<\alpha-1/2$ and $m\in\mathbb{N}$. Since $\lambda\asymp (\log n/n)^{\frac{2\alpha}{2\alpha+1}}$, we have $\tilde{\kappa}_\alpha^2\asymp\lambda^{-\frac 1{2\alpha}}\asymp ({n}/{\log n})^{\frac{1}{2\alpha+1}}=o(\sqrt{n/\log n})$. Besides, $\tilde{\kappa}_{\alpha,k}^2\asymp (n/\log n )^{\frac{2k+1}{2\alpha+1}} = O(n\epsilon_n^2/\log n)$. It also follows that $\hat{\kappa}_{\alpha,k+1}^2\asymp \lambda^{-\frac{2k+3}{2\alpha}}\asymp\left(n/\log n\right)^{\frac{2k+3}{2\alpha+1}}=O(n)$ for any $k<\alpha-3/2$ and $k\in\mathbb{N}_0$. Assumption (D) follows from Lemma~\ref{thm:minimax.diff.matern}. This completes the proof. 
\end{proof}

\begin{proof}[Proof of Theorem~\ref{thm:eb.contraction}]
Since $K(X,X)$ is non-negative definite, we have $u_i\geq 0$ for $1\leq i\leq n$. Note that $\sup_{x\in\mX}K(x,x)<\infty$ as $K$ is a continuous bivariate function on a compact support $\mX\times \mX$. Then we have $\sum_{i=1}^{n}u_i=\tr(K(X,X)) \leq n \kappa^2$ where $\kappa^2 = \sup_{x\in\mX}K(x,x)$. 
Let $\bm f_0=(f_0(X_1),\ldots,f_0(X_n))^T$. The MMLE $\hat{\sigma}_n^2$ is a quadratic form in $Y$. In view of the well-known formula for the expectation of quadratic forms (cf. Theorem 11.19 in \cite{schott2016matrix}), we obtain
\begin{equation}
\EE(\hat{\sigma}^2_n\mid X)=\lambda \sigma_0^2\tr([K(X,X)+n\lambda I_n]^{-1})+\lambda\bm f_0^T[K(X,X)+n\lambda I_n]^{-1}\bm f_0.
\end{equation}
Therefore,
\begin{align}
|\EE(\hat{\sigma}^2_n\mid X)-\sigma_0^2| &\leq \left|\lambda \sigma_0^2\tr([K(X,X)+n\lambda I_n]^{-1})-\sigma_0^2\right|+\lambda\bm f^T_0[K(X,X)+n\lambda I_n]^{-1}\bm f_0\\
\label{eq:mean.sigma} &\leq \left|n^{-1} \sigma_0^2\tr([(n\lambda)^{-1}K(X,X)+I_n]^{-1})-\sigma_0^2\right|+\lambda\bm f^T_0[K(X,X)+n\lambda I_n]^{-1}\bm f_0.
\end{align}
It follows that the first term is bounded by
\begin{equation}\label{eq:variance.1}
\begin{aligned}
&\left|n^{-1} \sigma_0^2\tr([(n\lambda)^{-1}K(X,X)+I_n]^{-1})-\sigma_0^2\right|\\
&\qquad\qquad= n^{-1}\sigma_0^2 \tr(I_n-[(n\lambda)^{-1}K(X,X)+I_n]^{-1})
=n^{-1}\sigma_0^2\sum_{i=1}^n\left(1-\frac{1}{u_i/n\lambda+1}\right)\\
&\qquad\qquad= n^{-1}\sigma_0^2\sum_{i=1}^{n}\frac{u_i}{n\lambda+u_i}.
\end{aligned}
\end{equation}

We next consider the second term in Equation~\eqref{eq:mean.sigma}. Then 
\begin{align}
\|f_{X, \lambda}\|_{\bbH}^2 &= \bm f_0^T [K(X,X)+n\lambda I_n]^{-1} K(X, X) [K(X,X)+n\lambda I_n]^{-1} \bm f_0\\
&=  \bm f_0^T [K(X,X)+n\lambda I_n]^{-1} \bm f_{X, \lambda},
\end{align}
where $\bm f_{X, \lambda} =: f_{X, \lambda}(X) = K(X, X) [K(X,X)+n\lambda I_n]^{-1} \bm f_0$ is the vector evaluating $f_{X, \lambda}$ on $X$. 
Therefore, 
\begin{align}
\|f_{X, \lambda}\|_{\bbH}^2 & = \bm f_0^T [K(X,X)+n\lambda I_n]^{-1} (\bm f_{X, \lambda} - \bm f_0 + \bm f_0) \\
& = \bm f_0^T [K(X,X)+n\lambda I_n]^{-1} (\bm f_{X, \lambda} - \bm f_0) + \bm f_0^T [K(X,X)+n\lambda I_n]^{-1} \bm f_0,
\end{align}
which implies that
\begin{align}
&\ \lambda\bm f_0^T [K(X,X)+n\lambda I_n]^{-1} \bm f_0 \\
\label{eq:quadratic}\leq &\ \lambda\|f_\lambda\|_{\bbH}^2 + \lambda\|f_{X, \lambda} - f_\lambda \|_{\bbH}^2 - \lambda\bm f_0^T [K(X,X)+n\lambda I_n]^{-1} (\bm f_{X, \lambda} - \bm f_0).
\end{align}
Rewrite $f_0$ as $f_0=L_K^{r}g$ for some $g=L_K^{-r}f_0\in \Ltwo$ and thus $f_i=\mu_i^{r}g_i$. Representing the function $g$ by $g=\sum_{i=1}^{\infty}  g_i\psi_i$, gives $f_\lambda = \sum_{i=1}^{\infty} \frac{\mu_i}{\mu_i+\lambda}\mu_i^{r} g_i\psi_i$. When $0 <r\leq 1/2$, we have
\begin{align}
\lambda\|f_\lambda\|_\bbH^2&=\lambda\sum_{i=1}^{\infty} \left(\frac{\mu_i}{\mu_i+\lambda}\mu_i^{r} g_i\right)^2\bigg /\mu_i\\
&=\lambda^{2r}\sum_{i=1}^{\infty} \left(\frac{\lambda}{\mu_i+\lambda}\right)^{1-2r}\left(\frac{\mu_i}{\mu_i+\lambda}\right)^{1+2r}g_i^2\\
\label{eq:quadratic.1}&\leq \lambda^{2r}\|L_K^{-r}f_0\|^2_{2}.
\end{align}
According to (14) in \cite{liu2020non}, it holds with $\PP_0^{(n)}$-probability at least $1-n^{-10}$ that
\begin{equation}
\|f_{X,\lambda}-f_\lambda\|_\bbH\leq\frac{\kappa \|f_0\|_\infty \sqrt{10 \log(3n)}}{\sqrt{n}\lambda}  \left(10 + \frac{4 \kappa\sqrt{10 \log(3n)}}{3 \sqrt{n\lambda}} \right),
\end{equation}
and thus 
\begin{equation}
\label{eq:quadratic.2}\lambda\|f_{X,\lambda}-f_\lambda\|^2_\bbH\leq\frac{\kappa^2 \|f_0\|^2_\infty \cdot 10 \log(3n)}{n\lambda}  \left(10 + \frac{4 \kappa\sqrt{10 \log(3n)}}{3 \sqrt{n\lambda}} \right)^2.
\end{equation}
Let $\lambda_{\max}(A)$ denote the largest eigenvalue of a matrix $A$. We have $\lambda_{\max}([K(X,X)+n\lambda I_n]^{-1})=(u_n+n\lambda)^{-1}\leq (n\lambda)^{-1}$. Hence,
\begin{align}
|\lambda \bm f_0^T [K(X,X)+n\lambda I_n]^{-1} (\bm f_{X, \lambda} - \bm f_0)| &\leq \lambda \cdot \lambda_{\max}([K(X,X)+n\lambda I_n]^{-1}) |\bm f_0^T(\bm f_{X, \lambda} - \bm f_0)|\\
& \leq \lambda (n\lambda)^{-1} n \|f_0\|_\infty \|f_{X,\lambda} - f_0\|_\infty\\
& \leq \|f_0\|_\infty (\|f_{X,\lambda} - f_\lambda\|_\infty + \|f_\lambda - f_0\|_\infty)\\
\label{eq:quadratic.3}& \leq \|f_0\|_\infty ( \kappa \|f_{X,\lambda} - f_\lambda\|_{\bbH} + \|f_\lambda - f_0\|_\infty).
\end{align}
Substituting Equations~\eqref{eq:quadratic.1}, \eqref{eq:quadratic.2} and \eqref{eq:quadratic.3} into Equation~\eqref{eq:quadratic}, there exists $R_n \lesssim \lambda^{2r} + \log n/(n\lambda) + \sqrt{\log n}/(\sqrt{n}\lambda) + \|f_\lambda - f_0\|_\infty$ that does not depend on $X$ such that with $\PP_0^{(n)}$-probability at least $1-n^{-10}$,
\begin{equation}\label{mean.second}
\lambda \bm f_0^T [K(X,X)+n\lambda I_n]^{-1} \bm f_0 \leq R_n.
\end{equation}

Combining Equations~\eqref{eq:variance.1} and \eqref{mean.second} gives that with $\PP_0^{(n)}$-probability at least $1-n^{-10}$,
\begin{equation}
|\EE(\hat{\sigma}^2_n\mid X)-\sigma_0^2|\leq 
n^{-1} \sigma_0^2 \sum_{i=1}^{n}\frac{u_i}{n\lambda+u_i} + R_n.
\end{equation}

We now bound the variance of $\hat{\sigma}^2_n$. Using the variance formula for quadratic forms (cf. Theorem 11.23 in \cite{schott2016matrix}), we have
\begin{align}
\Var(\hat{\sigma}^2_n\mid X)&=2\lambda^2\sigma_0^4\tr([K(X,X)+n\lambda I_n]^{-2})+4\lambda^2 \sigma_0^2\bm f^T_0[K(X,X)+n\lambda I_n]^{-2}\bm f_0\\
&\leq 2\lambda^2\sigma_0^4\cdot n(n\lambda)^{-2}+4\lambda^2 \sigma_0^2\cdot \lambda_{\max}([K(X,X)+n\lambda I_n]^{-2})\|\bm f_0\|_2^2\\
&\leq 2\sigma_0^4 n^{-1}+4\sigma_0^2\|f_0\|_\infty^2 n^{-1}.
\end{align}
Therefore, for any $X \in \mX^n$, it holds with $\PP_0^{(n)}$-probability at least $1-n^{-10}$ that
\begin{align}
\EE[(\hat{\sigma}^2_n-\sigma_0^2)^2\mid X] &=\Var(\hat{\sigma}^2_n\mid X)+[\EE(\hat{\sigma}^2_n\mid X)-\sigma_0^2)]^2\\
& \leq  n^{-1} \sigma_0^2 \sum_{i=1}^{n}\frac{u_i}{n\lambda+u_i} + R_n + 2\sigma_0^4 n^{-1}+4\sigma_0^2\|f_0\|_\infty^2 n^{-1},
\end{align}
which by assumptions implies that
\begin{equation}
\EE[(\hat{\sigma}^2_n-\sigma_0^2)^2] = \EE[\EE[(\hat{\sigma}^2_n-\sigma_0^2)^2\mid X]] = o(1).
\end{equation}
It thus follows that $\hat{\sigma}^2_n$ converges to $\sigma_0^2$ in $\PP_0^{(n)}$-probability by applying Chebyshev's inequality.

Now we prove Theorem~\ref{thm:deriv.contraction.l2} under the empirical Bayes scheme. Consider a shrinking neighborhood $\mathcal{B}_n = (\sigma_0^2 - r_n, \sigma_0^2 + r_n)$ of $\sigma_0^2$ with $r_n = o(1)$, which satisfies that $\PP_0^{(n)}(\hat{\sigma}^2_n\in\mathcal{B}_n)\rightarrow 1$ according to the arguments above. Conditional on $\mathcal{B}_n$, Lemma~\ref{thm:equivalent.sigma.deriv} gives
\begin{equation}
\|\tilde{V}^{k}_{n}\|_\infty \leq \frac{2(\sigma_0^2+o(1))\tilde{\kappa}_{k}^2}{n}.
\end{equation}
Then, all the established inequalities in the proof of Theorem~\ref{thm:deriv.contraction.l2} hold uniformly over $\sigma^2\in \mathcal{B}_n$. It follows that
\begin{equation}
\sup_{\sigma^2\in\mathcal{B}_n}\Pi_{n,k}\left(f: \|f^{(k)}-f_0^{(k)}\|_p\geq M_n \epsilon_n \mid \data\right)\rightarrow 0
\end{equation}
in $\PP_0^{(n)}$-probability for $p=2,\infty$, which directly implies that
\begin{equation}
\Pi_{n,k}\left(f: \|f^{(k)}-f_0^{(k)}\|_p\geq M_n \epsilon_n \mid \data\right) \Big |_{\sigma^2 = \hat\sigma^2_n}\rightarrow 0
\end{equation}
in $\PP_0^{(n)}$-probability for $p=2,\infty$. This completes the proof. 
\end{proof}

\begin{proof}[Proof of Corollary~\ref{thm:eb.example}]
We just need to verify the conditions in Theorem~\ref{thm:eb.contraction}. We first prove the case when $f_0 \in A^\gamma[0,1]$ and $K_\gamma$ is used. First, we have
\begin{equation}
\s \frac{f_i^2}{\mu_i} \asymp \s e^{2\gamma i} f_i^2 < \infty. 
\end{equation}
It is obvious that $n\lambda \asymp \log n \to \infty$ by noting that $\lambda \asymp \log n/n$, and thus the second condition holds in view of Remark 2. Lemma~\ref{lem:exp.mismatch.deterministic} with $\gamma_0 = \gamma$ shows that $\|f_\lambda - f_0\|_\infty \lesssim \sqrt{\lambda} = o(1)$.

For $f_0 \in W^\alpha[0,1]$ or $f_0 \in H^\alpha[0,1]$ with $K_\alpha$ being the kernel, we have

\begin{equation}
\s \frac{f_i^2}{\mu_i} \asymp \s i^{2\alpha} f_i^2 \leq \left(\s i^\alpha |f_i|\right)^2 <  \infty.
\end{equation}
The second condition is satisfied since $\lambda \asymp (\log n/n)^{\frac{2\alpha}{2\alpha+1}}$. Using Lemma~\ref{lem:poly.mismatch.deterministic} with $\alpha_0 = \alpha$, it holds that $\|f_\lambda - f_0\|_\infty \lesssim \sqrt{\lambda} = o(1)$. 
\end{proof}

\begin{proof}[Proof of Theorem~\ref{thm:lambda.mmle}]
We first prove part (b). Let $\Sigma = n^{-1} Y^T[K(X,X)+n\lambda I_n]^{-1}Y\cdot[K(X,X)+n\lambda I_n]$. Then the log marginal likelihood by substituting $\hat\sigma_n^2$ for $\sigma_0^2$ is
\begin{align}
\ell(\lambda\mid X, Y, \hat\sigma_n^2) &= -\frac{1}{2}\log\det(2\pi\Sigma)-\frac{1}{2}Y^T\Sigma^{-1}Y\\
&=-\frac{1}{2}\log\det(2\pi\Sigma)-\frac{n}{2}\\
&=-\frac{n}{2}\log(2\pi)-\frac{1}{2}\log\det(\Sigma)-\frac{n}{2}\\
&=-\frac{1}{2}\log\det(\Sigma) + \text{constant}.
\end{align}
The MMLE of $\lambda$ is given by
\begin{equation}
\hat \lambda_n = \underset{\lambda > 0}{\text{argmax}} \
\ell(\lambda \mid X, y, \hat\sigma_n^2).
\end{equation}
Taking the first derivative of $\ell(\lambda\mid X, Y, \hat\sigma_n^2)$ with respect to $\lambda$, it follows that
\begin{equation}\label{eq:derivative}
\frac{\partial \ell}{\partial \lambda}=-\frac{1}{2}\tr\left(\Sigma^{-1}\frac{\partial \Sigma}{\partial \lambda}\right).
\end{equation}
Note that
\begin{align}
\frac{\partial \Sigma}{\partial \lambda}&=n^{-1}Y^T\frac{\partial [K(X,X)+n\lambda I_n]^{-1}}{\partial \lambda}Y K(X,X)+\frac{\partial \lambda Y^T[K(X,X)+n\lambda I_n]^{-1}Y}{\partial \lambda}\\
&=n^{-1}\frac{\partial Y^T[K(X,X)+n\lambda I_n]^{-1}Y}{\partial \lambda} K(X,X)+Y^T[K(X,X)+n\lambda I_n]^{-1}Y I_n\\
&\quad +\lambda\cdot \frac{\partial Y^T[K(X,X)+n\lambda I_n]^{-1}Y}{\partial \lambda}I_n.
\end{align}
We further have
\begin{align}
&\frac{\partial Y^T[K(X,X)+n\lambda I_n]^{-1}Y}{\partial \lambda}\\
=\ &Y^T\frac{\partial [K(X,X)+n\lambda I_n]^{-1}}{\partial \lambda}Y\\
=\ &-Y^T [K(X,X)+n\lambda I_n]^{-1} \frac{\partial [K(X,X)+n\lambda I_n]}{\partial\lambda} [K(X,X)+n\lambda I_n]^{-1}Y\\
=\ &-nY^T [K(X,X)+n\lambda I_n]^{-2}Y.
\end{align}
Thus,
\begin{align}
\frac{\partial \Sigma}{\partial \lambda}
&= -Y^T [K(X,X)+n\lambda I_n]^{-2}Y K(X,X)+Y^T[K(X,X)+n\lambda I_n]^{-1}YI_n \\
\label{eq:derivative.1}&\quad -\lambda nY^T [K(X,X)+n\lambda I_n]^{-2}YI_n.
\end{align}
We also have
\begin{equation}\label{eq:derivative.2}
\Sigma^{-1}=n \left\{Y^T[K(X,X)+n\lambda I_n]^{-1}Y\right\}^{-1}\cdot[K(X,X)+n\lambda I_n]^{-1}.
\end{equation}
Combining Equations~\eqref{eq:derivative.1} and \eqref{eq:derivative.2}, we arrive at
\begin{align}
&\Sigma^{-1}\frac{\partial \Sigma}{\partial \lambda}\\
=\ & -\frac{nY^T [K(X,X)+n\lambda I_n]^{-2}Y}{Y^T[K(X,X)+n\lambda I_n]^{-1}Y} [K(X,X)+n\lambda I_n]^{-1}K(X,X)+n[K(X,X)+n\lambda I_n]^{-1}\\
&-\frac{n^2\lambda Y^T [K(X,X)+n\lambda I_n]^{-2}Y}{Y^T[K(X,X)+n\lambda I_n]^{-1}Y} [K(X,X)+n\lambda I_n]^{-1}\\
=\ &-\frac{nY^T [K(X,X)+n\lambda I_n]^{-2}Y}{Y^T[K(X,X)+n\lambda I_n]^{-1}Y} [K(X,X)+n\lambda I_n]^{-1}[K(X,X)\\
&+n\lambda I_n]+n[K(X,X)+n\lambda I_n]^{-1}\\
=\ &-\frac{nY^T [K(X,X)+n\lambda I_n]^{-2}Y}{Y^T[K(X,X)+n\lambda I_n]^{-1}Y}I_n +n[K(X,X)+n\lambda I_n]^{-1}.
\end{align}
Setting Equation~\eqref{eq:derivative} to zero, $\hat\lambda_n$ satisfies
\begin{equation}\label{eq:constraint}
\frac{n\lambda^2Y^T [K(X,X)+n\lambda I_n]^{-2}Y}{\lambda Y^T[K(X,X)+n\lambda I_n]^{-1}Y} =  \frac{\tilde\sigma_n^2}{\hat\sigma_n^2} = \lambda\tr([K(X,X)+n\lambda I_n]^{-1}).
\end{equation}
On one hand,
\begin{equation}
\frac{\tilde\sigma_n^2}{\hat\sigma_n^2} = 1 - \frac{\hat\sigma_n^2-\tilde\sigma_n^2}{\hat\sigma_n^2}.
\end{equation}
On the other hand,
\begin{equation}
\lambda\tr([K(X,X)+n\lambda I_n]^{-1}) = \lambda \sum_{i=1}^n \frac{1}{u_i + n\lambda} = \frac{1}{n}\sum_{i=1}^n \left(1-\frac{u_i}{u_i + n\lambda}\right) = 1-\frac{1}{n}\sum_{i=1}^n \frac{u_i}{u_i + n\lambda}.
\end{equation}
Thus,
\begin{equation}\label{eq:MMLE}
\frac{\hat\sigma_n^2-\tilde\sigma_n^2}{\hat\sigma_n^2} = \frac{1}{n}\sum_{i=1}^n \frac{u_i}{u_i + n\lambda}.
\end{equation}
In view of Lemma~\ref{lem:finite.sum}, with $\PP_0^{(n)}$-probability at least $1-n^{-10}$ we have
\begin{equation}
\frac{1}{n} \sum_{i=1}^n \frac{u_i}{u_i + n\lambda} \asymp n^{-1} \lambda^{-\frac{1}{2\alpha}},
\end{equation}
where we require $\lambda \gtrsim (n/\log n)^{-\alpha + \frac{1}{2}}$. Now we analyze the left-hand side of Equation~\eqref{eq:MMLE}. Note that
\begin{equation}
\s \frac{f_i^2}{\mu_i^{\frac{\alpha_0}{\alpha}}} \asymp \s i^{2\alpha_0}f_i^2 \leq \left( \s i^{\alpha_0}|f_i| \right)^2 < \infty.
\end{equation}
Beside, $n^{-1} \sum_{i=1}^n \frac{u_i}{u_i + n\lambda} = o(1)$ in $\PP_0^{(n)}$-probability if $\lambda \gtrsim n^{-2\alpha}$ and Lemma~\ref{lem:poly.mismatch.deterministic} ensures that $\|f_\lambda - f_0\|_\infty = o(1)$. Hence, Theorem~\ref{thm:eb.contraction} gives that $\hat\sigma_n^2 \rightarrow \sigma_0^2$ in $\PP_0^{(n)}$-probability. We only need to study the rate of
\begin{align}
\hat\sigma_n^2-\tilde\sigma_n^2 &= \lambda Y^T[K(X,X)+n\lambda I_n]^{-1} [K(X,X)+n\lambda I_n] [K(X,X)+n\lambda I_n]^{-1} Y \\
&\quad - n\lambda^2Y^T[K(X,X)+n\lambda I_n]^{-2}\\
&=\lambda Y^T[K(X,X)+n\lambda I_n]^{-1} K(X,X) [K(X,X)+n\lambda I_n]^{-1} Y.
\end{align}
Let $\hat f_n(\cdot) = K(\cdot, X)[K(X,X)+n\lambda I_n]^{-1} Y$. Then,
\begin{equation}
\|\hat f_n\|_{\bbH}^2 = Y^T [K(X,X)+n\lambda I_n]^{-1} K(X, X) [K(X,X)+n\lambda I_n]^{-1} Y,
\end{equation}
and thus
\begin{equation}
\hat\sigma_n^2-\tilde\sigma_n^2=\lambda\|\hat f_n\|_{\bbH}^2 \leq \lambda\|\hat f_n - f_\lambda\|_{\bbH}^2 + \lambda\| f_\lambda\|_{\bbH}^2.
\end{equation}
According to Theorem 4 in \cite{liu2020non}, it holds with probability at least $1-n^{-10}$ that
\begin{align}
\|\hat{f}_n-f_\lambda\|_\bbH &\leq\frac{\kappa \|f_0\|_\infty \sqrt{10\log(9n)}}{\sqrt{n}\lambda}  \left(10 + \frac{4 \kappa\sqrt{10\log(9n)}}{3 \sqrt{n\lambda}} \right)+\frac{C \kappa \sigma \sqrt{10\log(3n)}}{\sqrt{n} \lambda}.
\end{align}
Moreover,
\begin{align}
\lambda\|f_\lambda\|_\bbH^2&=\lambda\sum_{i=1}^{\infty} \left(\frac{\mu_i}{\mu_i+\lambda} f_i\right)^2 \bigg /\mu_i\\
&=\lambda\sum_{i=1}^{\infty} \frac{\mu_i^{1+\frac{\alpha_0}{\alpha}}}{(\mu_i+\lambda)^2}\frac{f_i^2}{\mu_i^{\frac{\alpha_0}{\alpha}}}\\
&=\lambda^{\frac{\alpha_0}{\alpha}}\sum_{i=1}^{\infty} \left(\frac{\lambda}{\mu_i+\lambda}\right)^{1-\frac{\alpha_0}{\alpha}}\left(\frac{\mu_i}{\mu_i+\lambda}\right)^{1+\frac{\alpha_0}{\alpha}}\frac{f_i^2}{\mu_i^{\frac{\alpha_0}{\alpha}}}\\
&\lesssim \lambda^{\frac{\alpha_0}{\alpha}}.
\end{align}
We next show that this is also the lower bound for $\widetilde{W}^{\alpha_0}$. For any $\delta > 0$, let $I_\lambda := c_\delta \lambda^{-1/(2\alpha)}$ for a constant $c_\delta > 0$ chosen large enough so that $M_\lambda := I_\lambda / 2 \ge N_\delta$ for all sufficiently small $\lambda$, where $N_{\delta}$ comes from the lower bound condition (6). This ensures that the lower bound assumption on $f_i^2$ holds for all $i \in [M_\lambda, I_\lambda]$.
Since $\mu_i \asymp i^{-2\alpha}$, for all $i \in [M_{\lambda}, I_{\lambda}],$  we have $\mu_i \asymp \lambda,$ yielding $  \frac{\mu_i}{(\mu_i + \lambda)^2}  \asymp \frac{1}{\lambda}.$
Hence,
\[
\lambda\|f_\lambda\|_\bbH^2 = \lambda \sum_{i=1}^\infty \frac{\mu_i f_i^2}{(\mu_i + \lambda)^2} 
\;\gtrsim\;
\lambda \sum_{i = M_\lambda}^{I_\lambda} f_i^2;
\]
here and throughout, summations over real-valued indices (e.g., $M_\lambda$ to $I_\lambda$) are interpreted as sums over the integer indices in that interval. Now consider each refined function class. For $f_0 \in \widetilde{W}^{\alpha_0}$, we assume $f_i^2 \ge C_\delta \, i^{-2\alpha_0 - 1 - \delta}$. Hence,
$\sum_{i = M_\lambda}^{I_\lambda}
f_i^2
\;\gtrsim\;
\sum_{i = M_\lambda}^{I_\lambda}
i^{-2\alpha_0-1-\delta}
\;\asymp\;
I_\lambda^{-2\alpha_0-\delta} \gtrsim 
\lambda^{\alpha_0/\alpha + \delta/(2\alpha)}.$ Since $\delta>0$ is arbitrary, $\lambda \| f_\lambda \|_{\mathbb H}^2
\;\gtrsim\;
\lambda^{\alpha_0/\alpha}$ up to logarithmic factors.  

Therefore, with $\PP_0^{(n)}$-probability at least $1-n^{-10}$, the $\hat\lambda_n$ satisfies
\begin{equation}
\hat\lambda_n^{\frac{\alpha_0}{\alpha}} \asymp n^{-1} \hat\lambda_n^{-\frac{1}{2\alpha}},
\end{equation}
ignoring logarithmic terms. Solving for $\hat\lambda_n$ yields the rate of order $n^{-\frac{2\alpha}{2\alpha_0+1}}$. It satisfies the condition that $\hat\lambda_n \gtrsim (n/\log n)^{-\alpha + \frac{1}{2}}$ and $\hat\lambda_n \gtrsim n^{-2\alpha}$ when $\alpha > \frac{2\alpha_0+1}{4\alpha_0-2}$.

Now we prove part (a) by starting from Equation~\eqref{eq:MMLE}. On one hand, by Lemma~\ref{lem:finite.sum}, it holds with $\PP_0^{(n)}$-probability at least $1-n^{-10}$ that
\begin{equation}
\frac{1}{n} \sum_{i=1}^n \frac{u_i}{u_i + n\lambda} \asymp n^{-1}\log\frac{1}{\lambda}.
\end{equation}
assuming $\lambda \gtrsim e^{-\sqrt{n/\log n}}$. On the other hand, for $f_0\in A^{\gamma_0}[0,1]$ and $K_\gamma$ we have
\begin{equation}
\s \frac{f_i^2}{\mu_i^{\frac{\gamma_0}{\gamma}}} \asymp \s e^{2\gamma_0 i} f_i^2 \leq \left( \s e^{\gamma_0 i}|f_i| \right)^2 < \infty.
\end{equation}
Beside, $n^{-1} \sum_{i=1}^n \frac{u_i}{u_i + n\lambda} = o(1)$ in $\PP_0^{(n)}$-probability if $\lambda \gtrsim e^{-n}$ and Lemma~\ref{lem:exp.mismatch.deterministic} gives $\|f_\lambda - f_0\|_\infty = o(1)$, which satisfies the conditions in Theorem~\ref{thm:eb.contraction}. Similarly, we can show that $\lambda\|f_\lambda\|_\bbH^2 \lesssim \lambda^{\frac{\gamma_0}{\gamma}}$.

We next show that the same lower bound holds for $\widetilde A^{\gamma_0}$ using similar arugments for $\widetilde W^{\alpha_0}$.
Fix $\delta>0$, and let
$I_\lambda := \frac{c_{\delta}'}{2\gamma}\log\frac{1}{\lambda}$ for a constant $c_\delta' > 0$ chosen large enough so that $M_\lambda := I_\lambda / 2 \ge N_\delta$ for all sufficiently small $\lambda$, where $N_{\delta}$ comes from the lower bound condition (6). Since $\mu_i \asymp e^{-2\gamma i}$, for all $i\in[M_\lambda,I_\lambda]$ we have
$\mu_i \asymp \lambda$, and therefore $\frac{\mu_i}{(\mu_i+\lambda)^2} \asymp \frac{1}{\lambda}.$
It follows that
\[
\lambda\|f_\lambda\|_{\bbH}^2
=
\lambda \sum_{i=1}^\infty \frac{\mu_i f_i^2}{(\mu_i+\lambda)^2}
\;\gtrsim\;
\lambda \sum_{i=M_\lambda}^{I_\lambda} f_i^2
\gtrsim
\sum_{i=M_\lambda}^{I_\lambda} e^{-2(\gamma_0+\delta)i}
\;\asymp\;
e^{-2(\gamma_0+\delta)M_\lambda}  
=
\lambda^{\gamma_0/\gamma + \delta/\gamma}.
\]
Since $\delta>0$ is arbitrary, we conclude that $
\lambda\|f_\lambda\|_{\bbH}^2
\;\gtrsim\;
\lambda^{\gamma_0/\gamma}$ up to logarithmic factors. 

Therefore, with $\PP_0^{(n)}$-probability at least $1-n^{-10}$, the MMLE of $\lambda$ should satisfy
\begin{equation}
\hat\lambda_n^{\frac{\gamma_0}{\gamma}} \asymp n^{-1},
\end{equation}
ignoring logarithmic terms. This yields the rate of order  $n^{-\frac{\gamma}{\gamma_0}}$ for $\hat\lambda_n$ up to logarithmic factors, which satisfies $\hat\lambda_n \gtrsim e^{-\sqrt{n/\log n}}$ and $\hat\lambda_n \gtrsim e^{-n}$.
\end{proof}

\begin{proof}[Proof of Theorem~\ref{thm:adaptive.exp}]
This is an immediate result of Theorem~\ref{thm:eb.contraction}. The proof procedure is similar to that used in Theorem~\ref{thm:exp.contraction.deriv} and Corollary~\ref{thm:eb.example}, by combining the rate of $\lambda$ in Theorem~\ref{thm:lambda.mmle}, the convergence rate in Lemma~\ref{lem:exp.mismatch}, and Lemma~\ref{lem:finite.sum}.
\end{proof}

\begin{proof}[Proof of Theorem~\ref{thm:adaptive.poly}]
This is an immediate result of Theorem~\ref{thm:eb.contraction}. The proof procedure is similar to that used in Theorem~\ref{thm:contraction.deriv} and Corollary~\ref{thm:eb.example}, by combining the rate of $\lambda$ in Theorem~\ref{thm:lambda.mmle}, the convergence rate in Lemma~\ref{lem:poly.mismatch}, and Lemma~\ref{lem:finite.sum}.
\end{proof}

\subsection{Proofs in Section~\ref{sec:non-asymptotic}}

\begin{proof}[Proof of Lemma~\ref{thm:equivalent.sigma.deriv}]
Let $K_{0k,x'}(\cdot)=\partial_{x'}^kK(\cdot,x')$ and 
$$\hat{K}_{0k,x'}(\cdot)=K(\cdot, X)[K(X, X) + n \lambda I_n]^{-1} K_{0k}(X, x').$$ 
It is easy to see that $\hat{K}_{0k,x'}$ is the solution to a noise-free KRR with observations $K_{0k,x'}$. Moreover, we have
\begin{equation}\label{eq:tilde.variance}
\begin{split}
|\sigma^{-2}n\lambda\tilde{V}^{k}_{n}(x')| &= \left|\partial^k_{x}(\hat{K}_{0k,x'}-K_{0k,x'})(x)|_{x=x'}\right|\\
&\leq \|\partial^k_{x}(\hat{K}_{0k,x'}-K_{0k,x'})\|_\infty \leq \tilde{\kappa}_{k}
\|\hat{K}_{0k,x'}-K_{0k,x'}\|_{\tilde{\bbH}}.
\end{split}
\end{equation}
According to Lemma~\ref{cor:h.tilde.noiseless}, it holds with $\PP_0^{(n)}$-probability at least $1 - n^{-10}$  that
\begin{equation}
\|\hat{K}_{0k,x'}-K_{0k,x'}\|_{\tilde{\bbH}}\leq 2\|L_{\tilde{K}}K_{0k,x'}-K_{0k,x'}\|_{\tilde{\bbH}} \leq 2\lambda \tilde{\kappa}_{k}.
\end{equation}
\end{proof}

\begin{proof}[Proof of Lemma~\ref{lem:differentiability.exp}]
According to the proof of Lemma 11 in \cite{liu2020non}, we have
\begin{equation}
\tilde{\kappa}_{\gamma,m}^2 \asymp \s \frac{i^{2m}\mu_i}{\lambda + \mu_i} = \hat{\kappa}_{\gamma,m}^2 \asymp \sum_{i=1}^{\infty}\frac{e^{-2\gamma i}i^{2m}}{\lambda + e^{-2\gamma i}} \asymp \int_{0}^{\infty} \frac{x^{2m}}{\lambda e^{2 \gamma x} + 1} dx.
\end{equation}
Letting $e^{-\eta} = \lambda \in (0, 1)$ where $\eta > 0$, it follows that
\begin{equation}
\tilde{\kappa}_{\gamma,m}^2\asymp  (2\gamma)^{-2m-1} \int_{0}^{\infty} \frac{t^{2m}}{\lambda e^t+1}dt = (2\gamma)^{-2m-1} (2m)! F_{2m}(\eta),
\end{equation}
where $F_{2m}$ is the Fermi-Dirac integral of order $2m$ \citep{mcdougall1938computation}.

According to Equation (7.3) in \cite{mcdougall1938computation}, as $\eta \rightarrow \infty$ (or $\lambda \rightarrow 0$), it holds that
\begin{equation}
F_{2m}(\eta) \rightarrow \frac{1}{2m + 1} \eta^{2m + 1} = \frac{1}{2m + 1}(-\log \lambda)^{2m + 1}. 
\end{equation}
Therefore,
\begin{equation}
\tilde{\kappa}_{\gamma,m}^2 \asymp \hat{\kappa}_{\gamma,m}^2 \asymp (-\log \lambda)^{2m + 1}.
\end{equation}
The differentiability of $\tilde{K}_\gamma$ directly follows from the boundedness of $\tilde{\kappa}^2_{\gamma,m}$ for sufficiently small $\lambda$.
\end{proof}

\begin{proof}[Proof of Lemma~\ref{lem:differentiability.matern}]

According to the proof of Lemma 11 in \cite{liu2020non}, it holds that
\begin{equation}
\tilde{\kappa}_{\gamma,m}^2 \asymp \s \frac{i^{2m}\mu_i}{\lambda + \mu_i} = \hat{\kappa}_{\gamma,m}^2 \asymp \sum_{i=1}^{\infty}\frac{i^{-2\alpha}i^{2m}}{\lambda + i^{-2\alpha}} \asymp \lambda^{-\frac{2m+1}{2\alpha}}.
\end{equation}
\end{proof}

\begin{proof}[Proof of Lemma~\ref{thm:minimax.diff.exp}]
Assuming the conditions in Lemma~\ref{lem:approx.error} holds and invoking Lemma~\ref{lem:exp.RKHS.derivative}, it holds with $\PP_0^{(n)}$-probability at least $1 - n^{-10}$ that
\begin{equation}
\|\hat{f}_n^{(k)}-f_0^{(k)}\|_2 \lesssim\|f_\lambda^{(k)}-f_0^{(k)}\|_2 + \tilde\kappa_{\gamma,k} \sqrt{\frac{\log n}{n}}.
\end{equation}
By Lemma~\ref{lem:differentiability.exp} and Lemma~\ref{lem:exp.mismatch.deterministic}, we obtain
\begin{equation}
\|\hat{f}_n^{(k)}-f_0^{(k)}\|_2 \lesssim \lambda^{\frac{\gamma_0}{2\gamma} + \frac{\epsilon}{2}}
+(-\log\lambda)^{k+\frac 12}\sqrt{\frac{\log n}{n}}.
\end{equation}
for any arbitrary small $\epsilon > 0$.

Taking $\gamma_0 = \gamma$, the preceding display does not have a closed-form minimizer for $\lambda$. We take $\lambda\asymp \log n/n$, which satisfies $\tilde{\kappa}_\gamma^2\asymp \log (n/\log n)=o(\sqrt{n/\log n})$ and $\|f_\lambda - f_0\|_\infty = o(1)$, i.e., the conditions in Lemma~\ref{lem:approx.error} are satisfied. The proof is completed by substituting $\lambda$.
\end{proof}

\begin{proof}[Proof of Lemma~\ref{thm:minimax.diff.matern}]
By combining Lemma~\ref{lem:approx.error}, Lemma 12 in \cite{liu2020non}, Lemma~\ref{lem:differentiability.matern} and Lemma~\ref{lem:poly.mismatch.deterministic}, we obtain with $\PP_0^{(n)}$-probability at least $1-n^{-10}$ there holds 
\begin{equation}
\|\hat f_n^{(k)} - f_0^{(k)}\|_2 \lesssim \lambda^{\frac{\alpha_0-k}{2\alpha}} + \lambda^{-\frac{2k+1}{4\alpha}} \sqrt{\frac{\log n}{n}}.
\end{equation}
Let $\alpha_0 = \alpha$ and the preceding upper bound is minimized with $\lambda\asymp ({\log n}/{n})^{\frac{2\alpha}{2\alpha+1}}$.
\end{proof}

\begin{proof}[Proof of Lemma~\ref{lem:exp.mismatch}]
Following the proof of Lemma~\ref{thm:minimax.diff.exp}, it holds that
\begin{equation}
\|\hat f_n^{(k)} - f_0^{(k)}\|_2 \lesssim \lambda^{\frac{\gamma_0}{2\gamma}+\frac{\epsilon}{2}} + (-\log\lambda)^{k+\frac{1}{2}} \sqrt{\frac{\log n}{n}},
\end{equation}
where $\epsilon>0$ is an arbitrarily small number. Noting that $\hat\lambda_n \gtrsim n^{-\frac{\gamma}{\gamma_0}}$ and $ \hat\lambda_n \lesssim (n/\log n)^{-\frac{\gamma}{\gamma_0}}$ given in Theorem~\ref{thm:lambda.mmle}, we obtain that
\begin{equation}
\|\hat f_n^{(k)} - f_0^{(k)}\|_2 \lesssim \frac{(\log n)^{k+1}}{\sqrt{n}}.
\end{equation}
\end{proof}

\begin{proof}[Proof of Lemma~\ref{lem:poly.mismatch}]
Based on the proof of Lemma~\ref{thm:minimax.diff.matern}, we have
\begin{equation}
\|\hat f_n^{(k)} - f_0^{(k)}\|_2 \lesssim \lambda^{\frac{\alpha_0-k}{2\alpha}} + \lambda^{-\frac{2k+1}{4\alpha}} \sqrt{\frac{\log n}{n}}.
\end{equation}
Substituting $\lambda = \hat \lambda_n \asymp n^{-\frac{2\alpha}{2\alpha_0+1}}$ given in Theorem~\ref{thm:lambda.mmle}, the preceding inequality becomes
\begin{equation}
\|\hat f_n^{(k)} - f_0^{(k)}\|_2 \lesssim n^{-\frac{\alpha_0-k}{2\alpha_0+1}} \sqrt{\log n}.
\end{equation}
\end{proof}

\subsection{Auxiliary technical results}

\begin{lemma}\label{lem:entropy.integral.deriv}
Under the conditions of Theorem~\ref{thm:deriv.contraction.l2}, it holds with $\PP_0^{(n)}$-probability at least $1 - n^{-10}$ that
\begin{equation}
\EE\left(\|f^{(k)}-\hat{f}^{(k)}_n\|_\infty \mid \data\right) \lesssim \tilde{\kappa}_{k}\sqrt{\frac{\log n}{n}}.
\end{equation}
\end{lemma}

\begin{lemma}\label{lem:exp.RKHS.derivative}
Let the RKHS induced by $K_\gamma$ and $\tilde{K}_\gamma$ be $\bbH_\gamma$ and $\tilde{\bbH}_\gamma$, respectively. If $\gamma>1/2$, then $f\in C^k[0,1]$ for any $k\in\mathbb{N}_0$ and $f\in\bbH_\gamma$. Moreover, these exists a constant $C>0$ that does not depend on $\lambda$ such that $\|f^{(k)}\|_2\leq C\tilde{\kappa}_\gamma^{-1}\tilde{\kappa}_{\gamma,k}\|f\|_{\tilde{\bbH}_\gamma}$ for any $f\in\bbH_\gamma$.
\end{lemma}

\begin{lemma}\label{lem:exp.mismatch.deterministic}
Suppose $f_0\in A^{\gamma_0}[0,1]$ and the kernel is chosen to be $K_\gamma$ for $\gamma \geq \gamma_0 > 0$. Then it holds
\begin{equation}
\|f_\lambda - f_0\|_\infty \lesssim \lambda^{\frac{\gamma_0}{2\gamma}}.
\end{equation}
Furthermore, for any $k\in\mathbb{N}_0$ and an arbitrarily small $\epsilon>0$ it holds
\begin{equation}
\|f_\lambda^{(k)}-f_0^{(k)}\|_2 \lesssim \lambda^{\frac{\gamma_0}{2\gamma} + \frac{\epsilon}{2}}.
\end{equation}
\end{lemma}

\begin{lemma}\label{lem:poly.mismatch.deterministic}
Suppose $f_0\in W^{\alpha_0}[0,1]$ or $f_0\in H^{\alpha_0}[0,1]$ and the kernel is chosen to be $K_\alpha$ for $\alpha \geq \alpha_0 > 1/2$. Then it holds that
\begin{equation}
\|f_\lambda - f_0\|_\infty \lesssim \lambda^{\frac{\alpha_0}{2\alpha}}.
\end{equation}
Furthermore, for any $k\in\mathbb{N}_0$ such that $\alpha \geq \alpha_0 > k+1/2$ it holds that
\begin{equation}
\|f_\lambda^{(k)} - f_0^{(k)}\|_2 \lesssim \lambda^{\frac{\alpha_0-k}{2\alpha}}.
\end{equation}
\end{lemma}

\begin{lemma}\label{cor:h.tilde.noiseless}
Suppose the observations are noiseless. Under Assumption (A), by choosing $\lambda$ such that $\tilde{\kappa}^2=o(\sqrt{n/\log n})$, it holds with $\PP_0^{(n)}$-probability at least $1 - n^{-10}$ that
\begin{equation}
\|\hat{f}_n-f_0\|_{\tilde{\bbH}}\leq 2\|f_\lambda-f_0\|_{\tilde{\bbH}}.
\end{equation}
\end{lemma}

\begin{lemma}\label{lem:finite.sum}
\begin{enumerate}[(a)]
\item For the kernel $K_\gamma$, it holds with $\PP_0^{(n)}$-probability at least $1-n^{-10}$ that
\begin{equation}
\sum_{i=1}^n \frac{u_i}{u_i + n\lambda} \asymp -\log\lambda,
\end{equation}
if the regularization parameter satisfies $\lambda \gtrsim e^{-\sqrt{n/\log n}}$.
\item For the kernel $K_\alpha$, it holds with $\PP_0^{(n)}$-probability at least $1-n^{-10}$ that
\begin{equation}
\sum_{i=1}^n \frac{u_i}{u_i + n\lambda} \asymp \lambda^{-\frac{1}{2\alpha}},
\end{equation}
if the regularization parameter satisfies $\lambda \gtrsim (n/\log n)^{-\alpha + \frac{1}{2}}$.
\end{enumerate}
\end{lemma}

\subsection{Proofs of auxiliary technical results}

\begin{proof}[Proof of Lemma~\ref{lem:entropy.integral.deriv}]
Define a centered GP posterior $\tilde{Z}^k$ as
\begin{equation}\label{eq:centered.GP.tilde}
\tilde{Z}^k=(f^{(k)}-\hat{f}^{(k)}_n)\mid \data\sim {\rm GP}(0, \tilde{V}_n^k),
\end{equation}
with psuedo-metric $\tilde{\rho}(x,x')=\sqrt{{\rm Var}(\tilde{Z}^k(x)-\tilde{Z}^k(x'))}$ for $x,x'\in[0,1]$. By Dudley’s entropy integral Theorem,
\begin{equation}
\EE \left( \|f^{(k)}-\hat{f}^{(k)}_n\|_\infty \mid \data \right) \leq \int_0^{\tilde{\rho}([0,1])} \sqrt{\log N(\epsilon,[0,1],\tilde{\rho})}d\epsilon.
\end{equation}
Recall that in Equation~\eqref{eq:tilde.variance} the posterior covariance can be expressed as the bias of a noise-free KRR estimator:
\begin{equation}
\sigma^{-2}n\lambda \tilde{V}^k_n(x,x')=\partial^k_x(K_{0k,x'}-\hat{K}_{0k,x'})(x).
\end{equation}
Hence, we have
\begin{align}
\tilde{\rho}(x,x')^2&={\rm Var}(\tilde{Z}^k(x),\tilde{Z}^k(x'))=\tilde{V}^k_n(x,x)+\tilde{V}^k_n(x',x')-2\tilde{V}^k_n(x,x')\\
&= \sigma^2 (n \lambda)^{-1} \{\partial^k_xK_{0k,x}(x) - \partial^k_x\hat{K}_{0k,x}(x)+\partial^k_{x'}K_{0k,x'}(x') - \partial^k_{x'}\hat{K}_{0k,x'}(x')\\
&\qquad\qquad\quad\ -2\partial^k_xK_{0k,x'}(x) +2\partial^k_x \hat{K}_{0k,x'}(x)\}\\
&= \sigma^2 (n \lambda)^{-1} \{\partial^k_x(K_{0k,x}-K_{0j,x'})(x) - \partial^k_{x'}(K_{0k,x}-K_{0k,x'})(x') \\
&\qquad\qquad\quad\ - \partial^k_x(\hat{K}_{0k,x}-\hat{K}_{0k,x'})(x)+ \partial^k_{x'}(\hat{K}_{0k,x}-\hat{K}_{0k,x'})(x')\}.
\end{align}
Let $g_{0k}=K_{0k,x}-K_{0k,x'}$ and $\hat{g}_{0k}=\hat{K}_{0k,x}-\hat{K}_{0k,x'}$. Then the preceding display implies
\begin{equation}
\sigma^{-2}n\lambda\tilde{\rho}(x,x')^2=\partial^k_x(g_{0k}-\hat{g}_{0k})(x)-\partial^k_{x'}(g_{0k}-\hat{g}_{0k})(x').
\end{equation}
By Lemma~\ref{cor:h.tilde.noiseless}, with $\PP_0^{(n)}$-probability at least $1 - n^{-10}$ it holds that
\begin{equation}\label{eq:rho.decomposition.tilde}
\sigma^{-2}n\lambda\tilde{\rho}(x,x')^2 \leq 2\|\partial^k(\hat{g}_{0k}-g_{0k})\|_\infty
\leq 2\tilde{\kappa}_{k}\|\hat{g}_{0k}-g_{0k}\|_{\tilde{\bbH}} \leq 4\tilde{\kappa}_{k}\|L_{\tilde{K}}g_{0k}-g_{0k}\|_{\tilde{\bbH}}.
\end{equation}
Substituting $g_{0k}=\sum_{i=1}^{\infty}\mu_i(\phi^{(k)}_i(x)-\phi^{(k)}_i(x'))\phi_i$ yields
\begin{equation}\label{eq:rho.euclid.tilde}
\begin{split}
\|L_{\tilde{K}}g_{0k}-g_{0k}\|_{\tilde{\bbH}}^2&=  \sum_{i=1}^{\infty}\left(\frac{\lambda\mu_i(\phi^{(k)}_i(x)-\phi^{(k)}_i(x'))}{\lambda+\mu_i}\right)^2\bigg/\frac{\mu_i}{\lambda+\mu_i}\\
&\leq \lambda^2 L_{k,\phi}^2 |x-x'|^{2} \s \frac{i^{2k+2}\mu_i}{\lambda+\mu_i}\\
&\leq  L_{k,\phi}^2 \lambda^2 \hat{\kappa}_{k+1}^2 |x-x'|^{2}.
\end{split}
\end{equation}
It follows from Equations~\eqref{eq:rho.decomposition.tilde} and \eqref{eq:rho.euclid.tilde} that
\begin{equation}\label{eq:tilde.rho}
\tilde{\rho}(x,x')\leq C\sqrt{\tilde{\kappa}_{k}\hat{\kappa}_{k+1}} n^{-1/2} |x-x'|^{1/2}.
\end{equation}
By the inequality in \citet[p.~529]{ghosal2017fundamentals}, we have
\begin{equation}
N(\epsilon,[0,1],\tilde{\rho})\leq N\left(\left(C^{-1}\epsilon\sqrt{n/\tilde{\kappa}_{k}\hat{\kappa}_{k+1}}\right)^{2},[0,1],|\cdot|\right)\lesssim \frac{1}{(\epsilon\sqrt{n/\tilde{\kappa}_{k}\hat{\kappa}_{k+1}})^{2}}.
\end{equation}
Since $\tilde{\kappa}_{k}\hat{\kappa}_{k+1} \lesssim \hat{\kappa}_{k}\hat{\kappa}_{k+1} \leq \hat{\kappa}_{k+1}^2 =O(n)$, we obtain $\log N(\epsilon,[0,1],\tilde{\rho})\lesssim \log(1/\epsilon)$. On the other hand, note that
\begin{align}
\|L_{\tilde{K}}g-g\|_{\tilde{\bbH}}^2&=  \sum_{i=1}^{\infty}\left(\frac{\lambda\mu_i(\phi_i^{(k)}(x)-\phi_i^{(k)}(x'))}{\lambda+\mu_i}\right)^2\bigg/\frac{\mu_i}{\lambda+\mu_i}\\
&\leq \sum_{i=1}^{\infty}\frac{2\lambda^2\mu_i(\phi_i^{(k)}(x)^2+\phi_i^{(k)}(x')^2)}{\lambda+\mu_i}\\
&\leq 4\lambda^2\tilde{\kappa}_{k}^2.
\end{align}
The preceding inequality combined with Equation~\eqref{eq:rho.decomposition.tilde} gives that
\begin{equation}
\tilde{\rho}(x,x')^2\leq 4\sigma^2(n\lambda)^{-1}\tilde{\kappa}_{k}\|L_{\tilde{K}}g-g\|_{\tilde{\bbH}}\leq 8\sigma^2 \tilde{\kappa}_{k}^2 /n,
\end{equation}
which implies $\tilde{\rho}([0,1])\lesssim \tilde{\kappa}_{k}/\sqrt{n}$. Therefore,
\begin{equation}
\int_0^{\tilde{\rho}([0,1])} \sqrt{\log N(\epsilon,[0,1],\tilde{\rho})}d\epsilon \lesssim \int_0^{\tilde{\kappa}_{k}/\sqrt{n}} \sqrt{\log(1/\epsilon)}d\epsilon \lesssim \tilde{\kappa}_{k}\sqrt{\frac{\log n}{n}}.
\end{equation}
\end{proof}

\begin{proof}[Proof of Lemma~\ref{lem:exp.RKHS.derivative}]
In view of Corollary 4.36 in \cite{steinwart2008support}, we have $\tilde{K}_\gamma\in C^{2k}([0,1]\times[0,1])$. This implies that $f\in C^k[0,1]$ for any $f \in \tilde{\bbH}_\gamma$, which is also true for $f \in \bbH_\gamma$ since $\bbH_\gamma$ and $\tilde{\bbH}_\gamma$ contain the same functions.

Now we prove the norm inequality. Let $f=\s f_i\psi_i$ where $\{\psi_i\}_{i=1}^\infty$ is the Fourier basis. Then $\|f^{(k)}\|_2^2\asymp \s (f_i i^k)^2$ for any $m\in\mathbb{N}_0$. It is equivalent to showing that
\begin{equation}
\tilde{\kappa}_\gamma^2 \cdot \s f_i^2i^{2k} \leq C\tilde{\kappa}_{\gamma,k}^2 \cdot \s f_i^2\frac{\lambda+\mu_i}{\mu_i},
\end{equation}
for some $C>0$. Hence, it suffices to show that for any $i\in\mathbb{N}$,
\begin{equation}
\tilde{\kappa}_\gamma^2 \cdot f_i^2i^{2k}\leq C \tilde{\kappa}_{\gamma,k}^2 \cdot f_i^2\frac{\lambda+\mu_i}{\mu_i}.
\end{equation}
In view of Lemma~\ref{lem:differentiability.matern}, we have $\tilde{\kappa}_{\gamma,k}^2 \asymp (-\log\lambda)^{2k+1}$ and $\tilde{\kappa}_\gamma^2 \asymp -\log\lambda$. Since $\mu_i\asymp e^{-2\gamma i}$, the preceding inequality becomes
\begin{equation} \label{eq:sufficient} 
i^{2k}\leq C (-\log\lambda)^{2k} (1+\lambda e^{2\gamma i}).
\end{equation}
The above sufficient condition in Equation~\eqref{eq:sufficient} trivially holds when $i\leq -\log \lambda$ by taking $C=1$. When $i > -\log \lambda$, it follows that
\begin{equation}
\frac{i^{2k}}{(-\log\lambda)^{2k} (1+\lambda e^{2\gamma i})} \leq    \frac{i^{2k}}{(-\log\lambda)^{2k} \lambda e^{2\gamma i}} \leq \frac{i^{2k} }{(-\log\lambda)^{2k}e^{-i} e^{2\gamma i}}.
\end{equation}
This is bounded above by some constant $C$ since $\gamma > 1/2$, leading to Equation~\eqref{eq:sufficient}. This completes the proof.
\end{proof}

\begin{proof}[Proof of Lemma~\ref{lem:exp.mismatch.deterministic}]
Note that $f_\lambda^{(k)} - f_0^{(k)} = -\s \frac{\lambda}{\lambda + \mu_i} f_i \psi_i^{(k)}$ and $\{\psi_i^{(k)} / (2\pi i)^k\}_{i=1}^\infty$ is also a Fourier basis. Hence,
\begin{align}
\|f_\lambda^{(k)} - f_0^{(k)}\|_2^2 & = \s \left( \frac{\lambda}{\lambda + \mu_i} f_i (2\pi i)^k \right)^2\\
& \lesssim \s \left( \frac{\lambda}{\lambda + \mu_i} \right)^2  i^{2k} \mu_i^{\frac{\gamma_0}{\gamma}} e^{2\gamma_0 i} f_i^2\\
& \lesssim \s \left( \frac{\lambda}{\lambda + \mu_i} \right)^2  \mu_i^{\frac{\gamma_0}{\gamma}+\epsilon} e^{2\gamma_0 i} f_i^2\\
&\lesssim \lambda^{\frac{\gamma_0}{\gamma}+\epsilon},
\end{align}
where $\epsilon>0$ is an arbitrarily small number.

By Cauchy-Schwarz inequality,
\begin{align}
\|f_\lambda - f_0\|_\infty^2 & \lesssim \left(\s \frac{\lambda}{\lambda + \mu_i} |f_i|\right)^2 = \left(\s \frac{\lambda}{\lambda + \mu_i} e^{-\gamma_0 i} e^{\gamma_0 i} |f_i|\right)^2\\
& \asymp \left(\s \frac{\lambda}{\lambda + \mu_i}  \mu_i^{\frac{\gamma_0}{2\gamma}} e^{\gamma_0 i} |f_i| \right)^2 \leq \s \left( \frac{\lambda}{\lambda + \mu_i}  \mu_i^{\frac{\gamma_0}{2\gamma}} \right)^2 \s e^{2\gamma_0 i} f_i^2 \\
&\lesssim \lambda^{\frac{\gamma_0}{\gamma}}\sum_{i=1}^{\infty} \left(\frac{\lambda}{\mu_i+\lambda}\right)^{2-\frac{\gamma_0}{\gamma}}\left(\frac{\mu_i}{\mu_i+\lambda}\right)^{\frac{\gamma_0}{\gamma}}\\
&\lesssim \lambda^{\frac{\gamma_0}{\gamma}}.
\end{align}
\end{proof}

\begin{proof}[Proof of Lemma~\ref{lem:poly.mismatch.deterministic}]
When $f_0\in W^{\alpha_0}[0,1]$, it follows that
\begin{align}
\|f_\lambda^{(k)} - f_0^{(k)}\|_2^2 & = \s \left( \frac{\lambda}{\lambda + \mu_i} f_i (2\pi i)^k \right)^2 \leq \s \left( \frac{\lambda}{\lambda + \mu_i} \right)^2 i^{2k-2\alpha_0}i^{2\alpha_0} f_i^2\\
& \asymp \s \left( \frac{\lambda}{\lambda + \mu_i} \right)^2  \mu_i^{\frac{\alpha_0-k}{\alpha}} i^{2\alpha_0} f_i^2\\
&=\lambda^{\frac{\alpha_0-k}{\alpha}}\sum_{i=1}^{\infty} \left(\frac{\lambda}{\mu_i+\lambda}\right)^{2-\frac{\alpha_0-k}{\alpha}}\left(\frac{\mu_i}{\mu_i+\lambda}\right)^{\frac{\alpha_0-k}{\alpha}} i^{2\alpha_0} f_i^2\\
&\lesssim \lambda^{\frac{\alpha_0-k}{\alpha}}.
\end{align}
By Cauchy-Schwarz inequality, we have
\begin{align}
\|f_\lambda - f_0\|_\infty^2 & \lesssim \left(\s \frac{\lambda}{\lambda + \mu_i} |f_i|\right)^2 = \left(\s \frac{\lambda}{\lambda + \mu_i} i^{-\alpha_0}i^{\alpha_0} |f_i|\right)^2\\
& \asymp \left(\s \frac{\lambda}{\lambda + \mu_i}  \mu_i^{\frac{\alpha_0}{2\alpha}} i^{\alpha_0} |f_i| \right)^2 \leq \s \left( \frac{\lambda}{\lambda + \mu_i}  \mu_i^{\frac{\alpha_0}{2\alpha}} \right)^2 \s i^{2\alpha_0} f_i^2 \\
&\lesssim \lambda^{\frac{\alpha_0}{\alpha}}\sum_{i=1}^{\infty} \left(\frac{\lambda}{\mu_i+\lambda}\right)^{2-\frac{\alpha_0}{\alpha}}\left(\frac{\mu_i}{\mu_i+\lambda}\right)^{\frac{\alpha_0}{\alpha}}\\
\label{eq:f.lambda.f.0}&\lesssim \lambda^{\frac{\alpha_0}{\alpha}}.
\end{align}
The same results can be shown for $f_0\in H^{\alpha_0}[0,1]$ by noting that $\s i^{2\alpha_0} f_i^2 \leq \left(\s i^{\alpha_0} |f_i|\right)^2$.
\end{proof}

\begin{proof}[Proof of Lemma~\ref{cor:h.tilde.noiseless}]
A noise-free version of Equation~\eqref{eq:KRR.thm2} by substituting $\sigma = 0$  yields that with $\PP_0^{(n)}$-probability at least $1-n^{-10}$,
\begin{equation}
\|\hat{f}_n-f_\lambda\|_{\tilde{\bbH}}\leq \frac{\tilde{\kappa}^{-1}C(n,\tilde{\kappa})}{1-C(n,\tilde{\kappa})}\|f_\lambda-f_0\|_\infty.
\end{equation}
Since $\tilde{\kappa}^2=o(\sqrt{n/\log n})$, for sufficiently large $n$ we have $C(n,\tilde{\kappa})\leq 1/2$, and therefore the preceding inequality becomes $\|\hat{f}_n-f_\lambda\|_{\tilde{\bbH}}\leq \tilde{\kappa}^{-1}\|f_\lambda-f_0\|_\infty$. Then we have
\begin{align}
\|\hat{f}_n-f_0\|_{\tilde{\bbH}}&\leq\|\hat{f}_n-f_\lambda\|_{\tilde{\bbH}}+\|f_\lambda-f_0\|_{\tilde{\bbH}}\\
&\leq \tilde{\kappa}^{-1}\|f_\lambda-f_0\|_\infty+\|f_\lambda-f_0\|_{\tilde{\bbH}}\\
&\leq 2\|f_\lambda-f_0\|_{\tilde{\bbH}},
\end{align}
where the last inequality follows from the fact that $f(x)=\left<f,\tilde{K}_{x}\right>_{\tilde{\bbH}}\leq \|f\|_{\tilde{\bbH}} \|K_{x}\|_{\tilde{\bbH}}= \sqrt{\tilde{K}(x,x)}\|f\|_{\tilde{\bbH}}$.
\end{proof}

\begin{proof}[Proof of Lemma~\ref{lem:finite.sum}]
We first prove part (b). Note that $\mu_i \asymp i^{-2\alpha}$. For sufficiently large $n$, we further have
\begin{equation}
\sum_{i=1}^n \frac{\mu_i}{\mu_i + \lambda} \asymp \sum_{i=1}^n \frac{i^{-2\alpha}}{i^{-2\alpha} + \lambda} = \sum_{i=1}^n \frac{1}{1+i^{2\alpha} \lambda}\asymp \sum_{i=1}^\infty \frac{1}{1+i^{2\alpha} \lambda} \asymp \lambda^{-\frac{1}{2\alpha}}.
\end{equation}
Now we consider
\begin{align}
\bigg| \sum_{i=1}^n \frac{u_i}{u_i + n\lambda} - \sum_{i=1}^n \frac{\mu_i}{\mu_i + \lambda} \bigg| \leq \sum_{i=1}^n \frac{\lambda |u_i - n\mu_i|}{(u_i + n\lambda)(\mu_i + \lambda)} \leq \sum_{i=1}^n \frac{ |u_i/n - \mu_i|}{\mu_i + \lambda}.
\end{align}
According to Theorem 3 and Theorem A.4 in \cite{braun2006accurate} with $\delta = n^{-10}$, for any $1\leq r\leq n$, with $\PP_0^{(n)}$-probability at least $1-n^{-10}$ it holds
\begin{equation}
|u_i/n - \mu_i| \lesssim \mu_i r \sqrt{ \frac{\log r + \log n}{n}} + r^{1-2\alpha}.
\end{equation} 
Let $N = \lfloor n^{\frac{2\alpha-1}{2\alpha}}\rfloor$ and $r = i^{\frac{2\alpha}{2\alpha-1}}$, we have $\sqrt{\log r + \log n} \leq \sqrt{2\log n}$ and
\begin{align}
\sum_{i=1}^N \frac{ |u_i/n - \mu_i|}{\mu_i + \lambda} &\lesssim \sqrt{\frac{\log n}{n}}\sum_{i=1}^N \frac{ i^{-2\alpha+\frac{2\alpha}{2\alpha-1}} }{i^{-2\alpha} + \lambda} + \sum_{i=1}^N \frac{ i^{-2\alpha }}{i^{-2\alpha} + \lambda}\\
&\lesssim \sqrt{\frac{\log n}{n}} \lambda^{-\frac{1}{2\alpha-1}-\frac{1}{2\alpha}} + \lambda^{-\frac{1}{2\alpha}} \lesssim \lambda^{-\frac{1}{2\alpha}},
\end{align}
where the last inequality follows from $\lambda \gtrsim (n/\log n)^{-\alpha + \frac{1}{2}}$. Letting $r = i$, we have
\begin{align}
\sum_{i=N}^n \frac{ |u_i/n - \mu_i|}{\mu_i + \lambda} \lesssim  \sqrt{\frac{\log n}{n}} \sum_{i=N}^n \frac{ i^{1-2\alpha}}{i^{-2\alpha} + \lambda} + \sum_{i=N}^n \frac{ i^{1-2\alpha }}{i^{-2\alpha} + \lambda} = o(1),
\end{align}
where the last equality follows from Cauchy's criterion for convergence for sufficiently large $n$. Hence,
\begin{equation}
\sum_{i=1}^n \frac{u_i}{u_i + n\lambda} \asymp \lambda^{-\frac{1}{2\alpha}}.
\end{equation}

Now we prove part (a). Note that $\mu_i \asymp e^{-2\gamma i}$. For sufficiently large $n$, by Lemma~\ref{lem:differentiability.exp} we have
\begin{equation}
\sum_{i=1}^n \frac{\mu_i}{\mu_i + \lambda} \asymp -\log\lambda.
\end{equation}
According to Theorem 3 and Theorem A.4 in \cite{braun2006accurate} with $\delta = n^{-10}$, for any $1\leq r\leq n$, with $\PP_0^{(n)}$-probability at least $1-n^{-10}$ it holds that
\begin{equation}
|u_i/n - \mu_i| \lesssim \mu_i r \sqrt{ \frac{\log r + \log n}{n}} + e^{-2\gamma r}.
\end{equation} 
Letting $r = i$, we have
\begin{align}
\sum_{i=1}^n \frac{ |u_i/n - \mu_i|}{\mu_i + \lambda} &\lesssim \sqrt{\frac{\log n}{n}}\sum_{i=1}^n \frac{e^{-2\gamma i} i}{e^{-2\gamma i} + \lambda} + \sum_{i=1}^n \frac{ e^{-2\gamma i }}{e^{-2\gamma i} + \lambda}\\
&\lesssim \sqrt{\frac{\log n}{n}} (-\log\lambda)^2 - \log\lambda \lesssim -\log\lambda,
\end{align}
given that $\lambda \gtrsim e^{-\sqrt{n/\log n}}$. Hence,
\begin{equation}
\sum_{i=1}^n \frac{u_i}{u_i + n\lambda} \asymp -\log\lambda.
\end{equation}
\end{proof}

\section{Additional simulation results} \label{sec:simulation1}

In this additional simulation, we compare our plug-in GP estimators with the inverse method proposed in \cite{holsclaw2013gaussian} using one of their simulation designs, where the authors perceived the plug-in estimator as suboptimal. We also compare the empirical Bayes strategy used in the main paper with a fully Bayesian approach by placing inverse Gamma and Gamma priors on $\sigma^2$ and $\lambda$, respectively, which do not show significant differences between the two treatments under the simulation settings. 

We consider the regression function $f_0 (x) = x \sin(x)/10$ given as one simulated example in \cite{holsclaw2013gaussian}. We generate $n = 100$ and 500 data points on a regular grid in $[0, 10]$, and add noise $\varepsilon_i\sim N(0, \sigma^2)$ with $\sigma = 0.1$, 0.2, and 0.3. The goal is to estimate $f_0'$. The implementations of the plug-in GP method and the B-spline prior are the same as in the main paper. For the inverse method, we follow the priors and general setup in \cite{holsclaw2013gaussian}, where a power-exponential kernel with $\alpha = 1.99$ is used.

\begin{table}[ht]
\centering
\renewcommand\arraystretch{1.5}
\caption{RMSE of estimating $f_0'$, averaged over 100 repetitions.  The first four rows are the plug-in GP prior with various kernels (Mat\'ern kernel, squared exponential kernel, second-order Sobolev kernel, and the selected kernel via cross-validation). The fourth and fifth methods are the random series prior using B-splines and the inverse method proposed in \cite{holsclaw2013gaussian}. The last row ``Inverse*'' excludes simulations ``Inverse'' produces zero estimates. Methods with the smallest RMSE in each column are boldfaced. Standard errors are provided in parentheses. \label{table}}
\vspace{0.2in}
\resizebox{\columnwidth}{!}{
\begin{tabular}{lcccccc}
\hline
& \multicolumn{3}{c}{$n = 100$}& \multicolumn{3}{c}{$n = 500$}\\ \cline{2-7} 
& \begin{tabular}[c]{@{}c@{}}$\sigma = 0.1$\end{tabular} & \begin{tabular}[c]{@{}c@{}}$\sigma = 0.2$\end{tabular} & \begin{tabular}[c]{@{}c@{}}$\sigma = 0.3$\end{tabular} & \begin{tabular}[c]{@{}c@{}}$\sigma = 0.1$\end{tabular} & \begin{tabular}[c]{@{}c@{}}$\sigma = 0.2$\end{tabular} & \begin{tabular}[c]{@{}c@{}}$\sigma = 0.3$\end{tabular} \\ \hline
Mat\'ern  & 0.09 (0.02) & 0.13 (0.03) & 0.17 (0.03)  & 0.06 (0.01) & 0.08 (0.02) & 0.10 (0.02) \\
SE      & 0.10 (0.02) & 0.16 (0.02) & 0.20 (0.03) & 0.07 (0.01) & 0.10 (0.01) & 0.12 (0.02) \\
Sobolev  & 0.06 (0.01)& 0.08 (0.01)& 0.10 (0.02) & 0.03 (0.01) & 0.05 (0.01)& 0.07 (0.01)\\
CV & 0.06 (0.01) & 0.10 (0.03) & 0.12 (0.05) & 0.03 (0.01) & 0.05 (0.01) & 0.07 (0.01)\\ 
Inverse & 0.11 (0.12) & 0.11 (0.10) & 0.13 (0.09) & 0.15 (0.11) & 0.12 (0.12) & 0.11 (0.12)\\ 
Inverse* & 0.07 (0.02) & 0.08 (0.03) & 0.10 (0.03) & 0.11 (0.04) & 0.08 (0.03) & 0.07 (0.02)\\ 
B-splines & 0.08 (0.01) & 0.11 (0.02) & 0.14 (0.03) & 0.04 (0.01) & 0.06 (0.01) & 0.08 (0.02)\\ \hline
\end{tabular}
}
\end{table}

Table~\ref{table} reports the RMSEs of $f_0'$ calculated at 100 equally spaced points in $[0, 10]$ over 100 repetitions. We notice that in some repetitions the inverse method produces a zero function as the estimate, and we exclude these simulations and denote the results as ``Inverse*''; this is probably due to the complicated hyperparameter tuning without clear guides, which hampers easy implementation and optimal performance. Moreover, the inverse method is restricted to one particular derivative order, and its generalization to other derivative orders is nontrivial. From Table~\ref{table}, we can see that the plug-in Gaussian process prior with the Sobolev kernel achieves the best performance among all scenarios in general. The inverse method and B-splines behave relatively well for $n=100$ and $n=500$, respectively. In summary, our GP method tends to produce better estimation for larger $n$ and smaller $\sigma$, i.e., when the signal-to-noise ratio is large, and we have found no scenario in the considered settings where the inverse method has significantly smaller RMSE relative to standard errors.

We next compare a fully Bayesian approach with the empirical Bayes strategy for the proposed method. We use the priors $\sigma^2\sim \text{Inverse-Gamma}(20, 1)$ and $\lambda \sim \text{Gamma}(1, 1000)$ (parameterized using the shape and rate parameters with mean $1/1000$) and draw $S=10000$ posterior samples using the Metropolis-Hastings algorithm; the final estimate is the average of 10000 posterior samples $\hat f_{n,s}'$ for $s=1,\ldots ,S$. Similarly, Table~\ref{table:3} reports its RMSEs calculated at 100 equally spaced points in $[0, 10]$ over 100 repetitions, where we also include the results of empirical Bayes shown in Table~\ref{table} to ease comparison. According to Table~\ref{table:3}, we can see that there is no significant difference between the two treatments.

\begin{table}[ht]
\centering
\renewcommand\arraystretch{1.5}
\caption{RMSEs for estimating $f_0'$ over 100 repetitions for empirical Bayes (EB) and fully Bayesian (FB) approaches. Standard errors are provided in parentheses. \label{table:3}}
\vspace{0.2in}
\resizebox{\columnwidth}{!}{
\begin{tabular}{lcccccc}
\hline
& \multicolumn{3}{c}{$n = 100$}& \multicolumn{3}{c}{$n = 500$}\\ \cline{2-7} 
& \begin{tabular}[c]{@{}c@{}}$\sigma = 0.1$\end{tabular} & \begin{tabular}[c]{@{}c@{}}$\sigma = 0.2$\end{tabular} & \begin{tabular}[c]{@{}c@{}}$\sigma = 0.3$\end{tabular} & \begin{tabular}[c]{@{}c@{}}$\sigma = 0.1$\end{tabular} & \begin{tabular}[c]{@{}c@{}}$\sigma = 0.2$\end{tabular} & \begin{tabular}[c]{@{}c@{}}$\sigma = 0.3$\end{tabular} \\ \hline
Mat\'ern (EB)  & 0.09 (0.02) & 0.13 (0.03) & 0.17 (0.03)  & 0.06 (0.01) & 0.08 (0.02) & 0.10 (0.02) \\
Mat\'ern (FB)   & 0.08 (0.01) & 0.12 (0.02) & 0.15 (0.03) & 0.05 (0.01) & 0.08 (0.01) & 0.09 (0.02) \\
SE (EB)     & 0.10 (0.02) & 0.16 (0.02) & 0.20 (0.03) & 0.07 (0.01) & 0.10 (0.01) & 0.12 (0.02) \\
SE (FB)      & 0.10 (0.02) & 0.16 (0.02) & 0.19 (0.03) & 0.06 (0.00) & 0.10 (0.01) & 0.13 (0.02) \\
Sobolev (EB)  & 0.06 (0.01)& 0.08 (0.01)& 0.10 (0.02) & 0.03 (0.01)& 0.05 (0.01)& 0.07 (0.01)\\
Sobolev (FB)  & 0.07 (0.01) & 0.08 (0.01) & 0.10 (0.02) & 0.05 (0.01) & 0.06 (0.01) & 0.07 (0.01)\\ \hline 
\end{tabular}
}
\end{table}
\end{appendices}

\newpage

\normalem
\bibliographystyle{apalike}
\bibliography{reference}

@article{gine2010confidence,
  title={Confidence bands in density estimation},
  author={Gin{\'e}, Evarist and Nickl, Richard},
  journal={Annals of Statistics},
  volume={38},
  number={2},
  pages={1122--1170},
  year={2010},
  publisher={Institute of Mathematical Statistics},
  mrnumber={MR2604707},
  doi={10.1214/09-AOS738}
}

@article{bull2012honest,
  title={Honest adaptive confidence bands and self-similar functions},
  author={Bull, Adam D.},
  journal={Electronic Journal of Statistics},
  volume={6},
  pages={1490--1516},
  year={2012},
  publisher={Institute of Mathematical Statistics},
  mrnumber={MR2988456},
  doi={10.1214/12-EJS728}
}

@article{szabo2015frequentist,
  title={Frequentist coverage of adaptive nonparametric Bayesian credible sets},
  author={Szab{\'o}, Botond Tibor and van der Vaart, Aad W and van Zanten, J Harry},
  journal={Annals of Statistics},
  volume={43},
  number={4},
  pages={1391--1428},
  year={2015},
  publisher={Institute of Mathematical Statistics},
  doi={10.1214/14-AOS1270}
}

@article{sc13aos,
  title={Local and global asymptotic inference in smoothing spline models},
  author={Shang, Zuofeng and Cheng, Guang},
  journal={Annals of Statistics},
  volume={41},
  number={5},
  pages={2608--2638},
  year={2013},
  publisher={Institute of Mathematical Statistics}
}

@article{Li+Ghosal:17,
archivePrefix = {arXiv},
arxivId = {1508.05847},
author = {Li, Meng and Ghosal, Subhashis},
eprint = {1508.05847},
issn = {00905364},
journal = {Annals of Statistics},
number = {5},
pages = {2190--2217},
title = {{Bayesian detection of image boundaries}},
volume = {45},
year = {2017}
}

@article{jiang2021variable,
  title={Variable selection consistency of {G}aussian process regression},
  author={Jiang, Sheng and Tokdar, Surya T},
  journal={The Annals of Statistics},
  volume={49},
  number={5},
  pages={2491--2505},
  year={2021},
  publisher={Institute of Mathematical Statistics}
}

@article{gelfand2016spatial,
  title={Spatial statistics and {G}aussian processes: {A} beautiful marriage},
  author={Gelfand, Alan E and Schliep, Erin M},
  journal={Spatial Statistics},
  volume={18},
  pages={86--104},
  year={2016},
  publisher={Elsevier}
}

@book{shi2011gaussian,
  title={Gaussian Process Regression Analysis for Functional Data},
  author={Shi, Jian Qing and Choi, Taeryon},
  year={2011},
  publisher={CRC Press}
}

@article{smale2005shannon,
	title={Shannon sampling {II}: Connections to learning theory},
	author={Smale, Steve and Zhou, Ding-Xuan},
	journal={Applied and Computational Harmonic Analysis},
	volume={19},
	number={3},
	pages={285--302},
	year={2005},
	publisher={Academic Press}
}

@article{yang2017frequentist,
	title={Frequentist coverage and sup-norm convergence rate in {G}aussian process regression},
	author={Yang, Yun and Bhattacharya, Anirban and Pati, Debdeep},
	journal={arXiv preprint arXiv:1708.04753},
	year={2017}
}

@book{rasmussen2006,
	author={Carl Edward Rasmussen and Christopher K.I. Williams},
	title={Gaussian Process for Machine Learning},
	publisher={The MIT Press},
	year={2006}
}

@book{cucker2007learning,
  title={Learning Theory: An Approximation Theory Viewpoint},
  author={Cucker, Felipe and Zhou, Ding Xuan},
  volume={24},
  year={2007},
  publisher={Cambridge University Press}
}

@book{wahba1990spline,
	title={Spline Models for Observational Data},
	author={Wahba, Grace},
	volume={59},
	year={1990},
	publisher={Siam}
}

@article{zhang2005learning,
	title={Learning bounds for kernel regression using effective data dimensionality},
	author={Zhang, Tong},
	journal={Neural Computation},
	volume={17},
	number={9},
	pages={2077--2098},
	year={2005},
	publisher={MIT Press}
}

@article{smale2007learning,
	title={Learning theory estimates via integral operators and their approximations},
	author={Smale, Steve and Zhou, Ding-Xuan},
	journal={Constructive Approximation},
	volume={26},
	number={2},
	pages={153--172},
	year={2007},
	publisher={Springer}
}

@book{ghosal2017fundamentals,
	title={Fundamentals of Nonparametric Bayesian Inference},
	author={Ghosal, Subhashis and van der Vaart, Aad W},
	volume={44},
	year={2017},
	publisher={Cambridge University Press}
}

@article{yoo2016supremum,
	title={Supremum norm posterior contraction and credible sets for nonparametric multivariate regression},
	author={Yoo, William Weimin and Ghosal, Subhashis},
	journal={The Annals of Statistics},
	volume={44},
	number={3},
	pages={1069--1102},
	year={2016},
	publisher={Institute of Mathematical Statistics}
}

@article{van2009adaptive,
	title={Adaptive {B}ayesian estimation using a {G}aussian random field with inverse {g}amma bandwidth},
	author={van der Vaart, Aad W and van Zanten, J Harry},
	journal={The Annals of Statistics},
	volume={37},
	number={5B},
	pages={2655--2675},
	year={2009},
	publisher={Institute of Mathematical Statistics}
}

@article{rivoirard2012bernstein,
	title={Bernstein--von {M}ises theorem for linear functionals of the density},
	author={Rivoirard, Vincent and Rousseau, Judith},
	journal={The Annals of Statistics},
	volume={40},
	number={3},
	pages={1489--1523},
	year={2012},
	publisher={Institute of Mathematical Statistics}
}

@article{bickel2003nonparametric,
	title={Nonparametric estimators which can be ``plugged-in''},
	author={Bickel, Peter J and Ritov, Ya'acov},
	journal={The Annals of Statistics},
	volume={31},
	number={4},
	pages={1033--1053},
	year={2003},
	publisher={Institute of Mathematical Statistics}
}

@article{castillo2015bernstein,
	title={A {B}ernstein--von {M}ises theorem for smooth functionals in semiparametric models},
	author={Castillo, Isma{\"e}l and Rousseau, Judith},
	journal={The Annals of Statistics},
	volume={43},
	number={6},
	pages={2353--2383},
	year={2015},
	publisher={Institute of Mathematical Statistics}
}

@article{castillo2013nonparametric,
	title={Nonparametric {B}ernstein--von {M}ises theorems in {G}aussian white noise},
	author={Castillo, Isma{\"e}l and Nickl, Richard},
	journal={The Annals of Statistics},
	volume={41},
	number={4},
	pages={1999--2028},
	year={2013},
	publisher={Institute of Mathematical Statistics}
}

@article{de2013semiparametric,
	title={Semiparametric {B}ernstein--von {M}ises for the error standard deviation},
	author={de Jonge, Ren{\'e} and van Zanten, J Harry},
	journal={Electronic Journal of Statistics},
	volume={7},
	pages={217--243},
	year={2013},
	publisher={The Institute of Mathematical Statistics and the Bernoulli Society}
}

@article{liu2020non,
	title={On the estimation of derivatives using plug-in kernel ridge regression estimators},
	author={Liu, Zejian and Li, Meng},
	journal={Journal of Machine Learning Research},
        volume={24},
        number={266},
        pages={1–37},
	year={2023}
}

@article{van2008rates,
	title={Rates of contraction of posterior distributions based on {G}aussian process priors},
	author={van der Vaart, Aad W and van Zanten, J Harry},
	journal={The Annals of Statistics},
	volume={36},
	number={3},
	pages={1435--1463},
	year={2008},
	publisher={Institute of Mathematical Statistics}
}

@article{van2011information,
	title={Information rates of nonparametric {G}aussian process methods},
	author={van der Vaart, Aad W and van Zanten, J Harry},
	journal={Journal of Machine Learning Research},
	volume={12},
	number={6},
	year={2011}
}

@article{stone1982optimal,
	title={Optimal global rates of convergence for nonparametric regression},
	author={Stone, Charles J},
	journal={The Annals of Statistics},
	pages={1040--1053},
	volume={10},
	number={4},
	year={1982},
	publisher={JSTOR}
}

@article{pati2015adaptive,
	title={Adaptive {B}ayesian inference in the {G}aussian sequence model using exponential-variance priors},
	author={Pati, Debdeep and Bhattacharya, Anirban},
	journal={Statistics \& Probability Letters},
	volume={103},
	pages={100--104},
	year={2015},
	publisher={Elsevier}
}

@article{dai2018derivative,
  title={Derivative principal component analysis for representing the time dynamics of longitudinal and functional data},
  author={Dai, Xiongtao and M{\"u}ller, Hans-Georg and Tao, Wenwen},
  journal={Statistica Sinica},
  volume={28},
  number={3},
  pages={1583--1609},
  year={2018},
  publisher={JSTOR}
}

@article{holsclaw2013gaussian,
  title={Gaussian process modeling of derivative curves},
  author={Holsclaw, Tracy and Sans{\'o}, Bruno and Lee, Herbert KH and Heitmann, Katrin and Habib, Salman and Higdon, David and Alam, Ujjaini},
  journal={Technometrics},
  volume={55},
  number={1},
  pages={57--67},
  year={2013},
  publisher={Taylor \& Francis}
}

@article{bhattacharya2014anisotropic,
  title={Anisotropic function estimation using multi-bandwidth {G}aussian processes},
  author={Bhattacharya, Anirban and Pati, Debdeep and Dunson, David},
  journal={The Annals of Statistics},
  volume={42},
  number={1},
  pages={352},
  year={2014},
  publisher={NIH Public Access}
}

@article{li2020comparing,
  title={Comparing and weighting imperfect models using {D}-probabilities},
  author={Li, Meng and Dunson, David B},
  journal={Journal of the American Statistical Association},
  volume={115},
  number={531},
  pages={1349--1360},
  year={2020},
  publisher={Taylor \& Francis}
}

@article{james2009functional,
  title={Functional linear regression that’s interpretable},
  author={James, Gareth M and Wang, Jing and Zhu, Ji},
  journal={The Annals of Statistics},
  volume={37},
  number={5A},
  pages={2083--2108},
  year={2009},
  publisher={Institute of Mathematical Statistics}
}

@article{wang2020functional,
  title={Functional group bridge for simultaneous regression and support estimation},
  author={Wang, Zhengjia and Magnotti, John and Beauchamp, Michael S and Li, Meng},
  journal={Biometrics},
  volume={79},
  number={2},
  pages={1226--1238},
  year={2023},
  publisher={Wiley Online Library}
}

@article{gelfand2003spatial,
  title={Spatial modeling with spatially varying coefficient processes},
  author={Gelfand, Alan E and Kim, Hyon-Jung and Sirmans, CF and Banerjee, Sudipto},
  journal={Journal of the American Statistical Association},
  volume={98},
  number={462},
  pages={387--396},
  year={2003},
  publisher={Taylor \& Francis}
}

@article{banerjee2003directional,
  title={Directional rates of change under spatial process models},
  author={Banerjee, Sudipto and Gelfand, Alan E and Sirmans, CF},
  journal={Journal of the American Statistical Association},
  volume={98},
  number={464},
  pages={946--954},
  year={2003},
  publisher={Taylor \& Francis}
}

@article{wang2016estimating,
	title={Estimating shape constrained functions using {G}aussian processes},
	author={Wang, Xiaojing and Berger, James O},
	journal={SIAM/ASA Journal on Uncertainty Quantification},
	volume={4},
	number={1},
	pages={1--25},
	year={2016},
	publisher={SIAM}
}

@book{stein2012interpolation,
  title={Interpolation of Spatial Data: Some Theory for Kriging},
  author={Stein, Michael L},
  year={2012},
  publisher={Springer Science \& Business Media}
}

@inproceedings{riihimaki2010gaussian,
	title={Gaussian processes with monotonicity information},
	author={Riihim{\"a}ki, Jaakko and Vehtari, Aki},
	booktitle={Proceedings of the Thirteenth International Conference on Artificial Intelligence and Statistics},
	pages={645--652},
	year={2010}
}

@article{yu2020bayesian,
  title={Bayesian inference for stationary points in {G}aussian process regression models for event-related potentials analysis},
  author={Yu, Cheng-Han and Li, Meng and Noe, Colin and Fischer-Baum, Simon and Vannucci, Marina},
  journal={Biometrics},
  volume={79},
  number={2},
  pages={629--641},
  year={2023},
  publisher={Wiley Online Library}
}

@article{church2011sea,
  title={Sea-level rise from the late 19th to the early 21st century},
  author={Church, John A and White, Neil J},
  journal={Surveys in Geophysics},
  volume={32},
  number={4},
  pages={585--602},
  year={2011},
  publisher={Springer}
}

@article{church200620th,
  title={A 20th Century Acceleration in Global Sea-level Rise},
  author={Church, John A and White, Neil J},
  journal={Geophysical Research Letters},
  volume={33},
  number={1},
  year={2006},
  publisher={Wiley Online Library}
}

@book{berlinet2011reproducing,
  title={Reproducing Kernel Hilbert Spaces in Probability and Statistics},
  author={Berlinet, Alain and Thomas-Agnan, Christine},
  year={2011},
  publisher={Springer Science \& Business Media}
}

@article{cahill2015modeling,
  title={Modeling sea-level change using errors-in-variables integrated {G}aussian processes},
  author={Cahill, Niamh and Kemp, Andrew C and Horton, Benjamin P and Parnell, Andrew C},
  journal={The Annals of Applied Statistics},
  volume={9},
  number={2},
  pages={547--571},
  year={2015},
  publisher={Institute of Mathematical Statistics}
}

@article{wahba1990optimal,
  title={When is the optimal regularization parameter insensitive to the choice of the loss function?},
  author={Wahba, Grace and Wang, Yonghua},
  journal={Communications in Statistics-Theory and Methods},
  volume={19},
  number={5},
  pages={1685--1700},
  year={1990},
  publisher={Taylor \& Francis}
}

@article{charnigo2011generalized,
  title={A generalized ${C}_p$ criterion for derivative estimation},
  author={Charnigo, Richard and Hall, Benjamin and Srinivasan, Cidambi},
  journal={Technometrics},
  volume={53},
  number={3},
  pages={238--253},
  year={2011},
  publisher={Taylor \& Francis}
}

@book{wasserman2006all,
  title={All of Nonparametric Statistics},
  author={Wasserman, Larry},
  year={2006},
  publisher={Springer Science \& Business Media}
}

@article{santin2016approximation,
  title={Approximation of eigenfunctions in kernel-based spaces},
  author={Santin, Gabriele and Schaback, Robert},
  journal={Advances in Computational Mathematics},
  volume={42},
  number={4},
  pages={973--993},
  year={2016},
  publisher={Springer}
}

@article{seeger2008information,
  title={Information consistency of nonparametric Gaussian process methods},
  author={Seeger, Matthias W and Kakade, Sham M and Foster, Dean P},
  journal={IEEE Transactions on Information Theory},
  volume={54},
  number={5},
  pages={2376--2382},
  year={2008},
  publisher={IEEE}
}

@article{dong2006combined,
  title={A combined direct numerical simulation--particle image velocimetry study of the turbulent near wake},
  author={Dong, S and Karniadakis, GE and Ekmekci, A and Rockwell, D},
  journal={Journal of Fluid Mechanics},
  volume={569},
  pages={185--207},
  year={2006},
  publisher={Cambridge University Press}
}

@article{dong2008turbulent,
  title={Turbulent flow between counter-rotating concentric cylinders: a direct numerical simulation study},
  author={Dong, S},
  journal={Journal of Fluid Mechanics},
  volume={615},
  pages={371--399},
  year={2008},
  publisher={Cambridge University Press}
}

@article{sani1994resume,
  title={R{\'e}sum{\'e} and remarks on the open boundary condition minisymposium},
  author={Sani, Robert L and Gresho, Philip M},
  journal={International Journal for Numerical Methods in Fluids},
  volume={18},
  number={10},
  pages={983--1008},
  year={1994},
  publisher={Wiley Online Library}
}

@article{dong2014robust,
  title={A robust and accurate outflow boundary condition for incompressible flow simulations on severely-truncated unbounded domains},
  author={Dong, Suchuan and Karniadakis, George E and Chryssostomidis, Chryssostomos},
  journal={Journal of Computational Physics},
  volume={261},
  pages={83--105},
  year={2014},
  publisher={Elsevier}
}

@article{nussbaum1985spline,
  title={Spline smoothing in regression models and asymptotic efficiency in L2},
  author={Nussbaum, Michael},
  journal={The Annals of Statistics},
  pages={984--997},
  year={1985},
  publisher={JSTOR}
}

@article{castillo2018empirical,
  title={Empirical {B}ayes analysis of spike and slab posterior distributions},
  author={Castillo, Isma{\"e}l and Mismer, Romain},
  journal={Electronic Journal of Statistics},
  volume={12},
  pages={3953--4001},
  year={2018}
}

@article{castillo2020spike,
  title={Spike and slab empirical {B}ayes sparse credible sets},
  author={Castillo, Isma{\"e}l and Szab{\'o}, Botond Tibor},
  journal={Bernoulli},
  volume={26},
  number={1},
  pages={127--158},
  year={2020}
}

@article{szabo2013empirical,
  title={Empirical {B}ayes scaling of {G}aussian priors in the white noise model},
  author={Szab{\'o}, Botond Tibor and van der Vaart, Aad W and van Zanten, J Harry},
  journal={Electronic Journal of Statistics},
  volume={7},
  pages={991--1018},
  year={2013}
}

@article{rousseau2017asymptotic,
  title={Asymptotic behaviour of the empirical {B}ayes posteriors associated to maximum marginal likelihood estimator},
  author={Rousseau, Judith and Szab{\'o}, Botond Tibor},
   journal={Annals of Statistics},
   pages={833--865},
   volume={45},
   number=2,
  year={2017}
}

@article{fischer2020sobolev,
  title={Sobolev norm learning rates for regularized least-squares algorithms},
  author={Fischer, Simon and Steinwart, Ingo},
  journal={Journal of Machine Learning Research},
  volume={21},
  number={205},
  pages={1--38},
  year={2020}
}

@article{blanchard2018optimal,
  title={Optimal rates for regularization of statistical inverse learning problems},
  author={Blanchard, Gilles and M{\"u}cke, Nicole},
  journal={Foundations of Computational Mathematics},
  volume={18},
  number={4},
  pages={971--1013},
  year={2018},
  publisher={Springer}
}

@article{giordano2022nonparametric,
  title={Nonparametric {B}ayesian inference for reversible multidimensional diffusions},
  author={Giordano, Matteo and Ray, Kolyan},
  journal={The Annals of Statistics},
  volume={50},
  number={5},
  pages={2872--2898},
  year={2022},
  publisher={Institute of Mathematical Statistics}
}

@book{nickl2023bayesian,
  title={Bayesian Non-linear Statistical Inverse Problems},
  author={Nickl, Richard},
  year={2023},
  publisher={EMS press}
}

@article{knapik2011bayesian,
  title={Bayesian inverse problems with {G}aussian priors},
  author={Knapik, BT and van der Vaart, AW and van Zanten, J Harry},
  journal={Annals of Statistics},
  volume={39},
  number={5},
  pages={2626--2657},
  year={2011},
  publisher={Institute of Mathematical Statistics}
}

@book{van1996weak,
  title={Weak Convergence and Empirical Processes: With Applications to Statistics},
  author={van der Vaart, A.W. and Wellner, J.},
  isbn={9780387946405},
  lccn={95049099},
  series={Springer Series in Statistics},
  url={https://books.google.com/books?id=OCenCW9qmp4C},
  year={1996},
  publisher={Springer}
}

@book{schott2016matrix,
	title={Matrix Analysis for Statistics},
	author={Schott, James R},
	year={2016},
	publisher={John Wiley \& Sons}
}

@article{mcdougall1938computation,
  title={The computation of Fermi-Dirac functions},
  author={McDougall, J and Stoner, Edmund Clifton},
  journal={Philosophical Transactions of the Royal Society of London. Series A, Mathematical and Physical Sciences},
  volume={237},
  number={773},
  pages={67--104},
  year={1938},
  publisher={The Royal Society London}
}

@book{steinwart2008support,
	title={Support Vector Machines},
	author={Steinwart, Ingo and Christmann, Andreas},
	year={2008},
	publisher={Springer Science \& Business Media}
}

@article{braun2006accurate,
  title={Accurate error bounds for the eigenvalues of the kernel matrix},
  author={Braun, Mikio L},
  journal={The Journal of Machine Learning Research},
  volume={7},
  pages={2303--2328},
  year={2006},
  publisher={JMLR. org}
}

\end{document}